\documentclass[11pt]{article}
\usepackage{amssymb,latexsym,amsmath,amsbsy,amsthm,amsxtra,amsgen,graphicx,makeidx,dsfont,longtable}
\usepackage{bbm}
\usepackage{indentfirst}
\usepackage{geometry}

\usepackage{cite}
\numberwithin{equation}{section}

\usepackage{float}
\usepackage{subfigure}
\usepackage{tikz}
\usetikzlibrary{shapes.misc}
\usetikzlibrary{arrows.meta}
\tikzset{cross/.style={cross out, draw=black, minimum size=2*(#1-\pgflinewidth), inner sep=0pt, outer sep=0pt},
	cross/.default={1pt}}

\newfont{\msbm}{msbm10 at 11pt}

\newfont{\msbmsm}{msbm10 at 8pt}

\newtheorem{theorem}{Theorem}[section]

\newtheorem{lemma}[theorem]{Lemma}
\newtheorem{proposition}[theorem]{Proposition}

\theoremstyle{definition}
\newtheorem{remark}[theorem]{Remark}
\newgeometry{left = 2 cm, top=2cm, right = 2cm, bottom=2cm} 

\def\E{\mathbb{E}}
\def\Var{\textup{Var}}
\def\Cov{\textup{Cov}}

\begin{document}
\title{Asymptotics for the site frequency spectrum associated with the genealogy of a birth and death process}
\author{Jason Schweinsberg and Yubo Shuai \\
University of California at San Diego}
\maketitle



\vspace{-.6in}
\begin{abstract}
Consider a birth and death process started from one individual in which each individual gives birth at rate $\lambda$ and dies at rate $\mu$, so that the population size grows at rate $r = \lambda - \mu$.  Lambert \cite{lambert2018coalescent} and Harris, Johnston, and Roberts \cite{harris2020coalescent} came up with methods for constructing the exact genealogy of a sample of size $n$ taken from this population at time $T$.  We use the construction of Lambert, which is based on the coalescent point process, to obtain asymptotic results for the site frequency spectrum associated with this sample.  In the supercritical case $r > 0$, our results extend results of Durrett \cite{durrett2013population} for exponentially growing populations.  In the critical case $r = 0$, our results parallel those that Dahmer and Kersting \cite{dahmer2015internal} obtained for Kingman's coalescent.
\end{abstract}

\section{Introduction}

In mathematical population genetics, one is often interested in understanding the genealogy of a sample from a population.  That is, we take a sample from a population at some time and then trace the ancestral lines backwards in time.  As we trace back the ancestral lines, the lineages will merge, until eventually all of the sampled individuals are traced back to one common ancestor, leading to what is known as a coalescent process.  
The simplest coalescent process is Kingman's coalescent \cite{kingman1982coalescent}, in which each pair of lineages merges at rate 1.  It is well-known that Kingman's coalescent models the genealogy of many populations with constant size.

In this paper, we will consider instead a population with varying size which evolves according to a critical or supercritical birth and death processes in continuous time.  More specifically, we begin with one individual at time zero, and then each individual independently gives birth to a new individual at rate $\lambda$ and dies at rate $\mu$. We write $r=\lambda-\mu$ for the net growth rate and call the birth and death process supercritical if $r>0$ and critical if $r=0$.  For $t \geq 0$, we denote the population size at time $t$ by $N_t$. We then choose a random sample of $n$ individuals from the population at some time $T$, conditional on $N_T\ge n$. See Figure \ref{Figure: genealogical tree} for an illustration of the genealogical tree of a sample.

\begin{figure}[h]
    \centering
    \begin{tikzpicture}
        \draw [very thick] (3,0)--(3.75,1);
        \draw [very thick] (4.5,0)--(3.75,1);
        
        \draw [very thick] (7.5,0)--(8.25,1.5);
        \draw [dotted, very thick] (9,0)--(8.25,1.5);
        
        \draw [dotted, very thick] (6,0)--(4.5,2);
        \draw [very thick] (3.75,1)--(4.5,2);
        
        \draw [very thick] (10.5,0)--(9,2.5);
        \draw [very thick] (8.25,1.5)--(9,2.5);
        
        \draw [very thick] (4.5,2)--(7,4.5);
        \draw [very thick] (9,2.5)--(7,4.5);
        
        \draw [dotted, very thick] (5.5,3)--(6.5,2);
        \draw [dotted, very thick] (6,2.5)--(5.25,1.75);
        \draw [dotted, very thick] (8,3.5)--(7.5,3);
        \draw [dotted, very thick] (9.75,1.25)--(9.5,0.75);

        \draw (6.5,2) node[cross=4pt,very thick] {};
        \draw (5.25,1.75) node[cross=4pt,very thick] {};
        \draw (7.5,3) node[cross=4pt,very thick] {};
        \draw (9.5,0.75) node[cross=4pt,very thick] {};
        
        \node at (3,-0.5){1};
        \node at (4.5,-0.5){2};
        \node at (6,-0.5){3};
        \node at (7.5,-0.5){4};
        \node at (9,-0.5){5};
        \node at (10.5,-0.5){6};
        
        
        
        
        
        
        \node at (10.75,0.25){$G$};
        \node at (7.25,0.25){$F$};
        \node at (8.5,2.5){$E$};
        \node at (4.75,0.25){$D$};
        \node at (2.75,0.25){$C$};
        \node at (3.5,1.25){$B$};
        \node at (6.75,4.75){$A$};
    \end{tikzpicture}
    \caption{The genealogical tree of a sample of size $n = 4$ from a population of size $N_T = 6$.  Crosses represent deaths and dotted lines represent individuals that are not sampled.} \label{Figure: genealogical tree}
\end{figure}
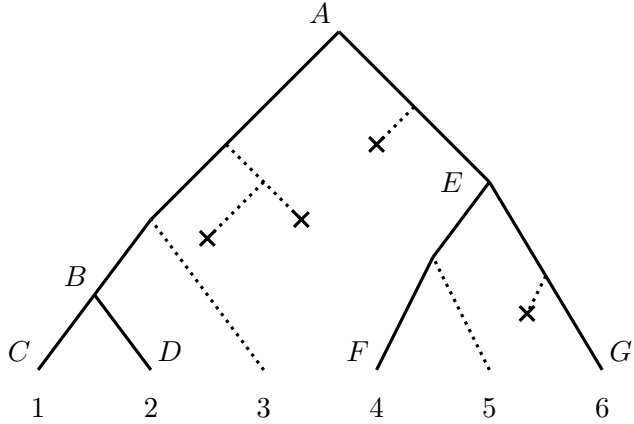

To understand the genealogical tree, we need to understand the coalescence times, which are the times when two branches of the tree merge because the sampled individuals share a common ancestor.  The study of the coalescence times in the genealogy of a branching process was pioneered by Fleischmann and Siegmund-Schultze in \cite{Fleischmann1977TheSO}. They considered a critical Galton-Watson process where all the individuals alive at time $T$ are sampled. They showed that the coalescent times can be approximated by the coalescent times of a Yule process under a deterministic time change. In \cite{OConnell1995TheGO}, O'Connell generalized their results to near-critical Galton-Watson processes, where the expected number of offspring is $1+\alpha/T$ for a fixed constant $\alpha$.

Understanding the genealogy of a sample of size $n$ from a birth and death process is more complicated. 
 While preliminary results were obtained in \cite{zub76, durr78, le14, gh18}, a full description of the genealogy of a sample was obtained only within the last few years.  Harris, Johnston and Roberts \cite{harris2020coalescent} considered the coalescent times in a uniform sample of size $n$ from a continuous-time Galton-Watson process. For birth and death processes, they obtained an exact formula for the joint density of the coalescence times. Lambert \cite{lambert2018coalescent} obtained the same joint density of the coalescent times in a birth and death process using the coalescent point process, an approach which we will describe in detail in Section \ref{Section: Coalescent point process}. Johnston \cite{johnston2019genealogy} generalized this formula to Galton-Watson processes with general offspring distributions.  Harris, Johnston, and Pardo \cite{hjp23} extended the asymptotic results in \cite{harris2020coalescent} to the case of heavy-tailed offspring distributions.  Johnston and Lambert \cite{jl23} considered multitype branching processes, and Harris, Palau, and Pardo \cite{hpp22} considered varying environments.  Burden and Griffiths \cite{bg22} obtained some closely related results for the genealogy of a sample from a Feller diffusion, which arises as a scaling limit of nearly critical birth and death processes.

We are interested in the distribution of the branch lengths in the genealogical tree of a sample from a birth and death process.
We write $L^k_{n,T}$ for the total length of the branches that support exactly $k$ leaves in the genealogical tree. For example, in Figure \ref{Figure: genealogical tree}, the branch from $A$ to $B$ supports two leaves, namely 1 and 2. The branch from $B$ to $C$ supports only one leaf, namely 1. In general, writing, for example, $AB$ for the length of the branch from $A$ to $B$, we have $L^1_{n,T}=BC+BD+EF+EG$, $L^2_{n,T}=AB+AE$, and $L^3_{n,T}=L^4_{n,T}=0$.  Branches supporting only one leaf are called external, while other branches are called internal.

These branch lengths are of interest because a mutation that occurs on a branch that supports $k$ leaves in the genealogical tree will be inherited by $k$ of the $n$ individuals in the sample. Let $M^k_{n,T}$ be the number of mutations inherited by exactly $k$ individuals in the sample. A commonly used summary statistic for genetic data is the site frequency spectrum, which consists of the numbers $M^1_{n,T}, \dots, M^{n-1}_{n,T}$.  Suppose that, independently of the birth and death process, mutations occur at a constant rate $\nu$ along each branch of the genealogical tree and all mutations are distinct. Then the conditional distribution of $M^k_{n,T}$ given $L^k_{n,T}$ is Poisson with mean $\nu L^k_{n,T}$. Consequently, to understand the asymptotic behavior of the site frequency spectrum, it suffices to understand the asymptotic behavior of the branch lengths. We will therefore focus on the branch lengths $L^k_{n,T}$ in this paper.


\subsection{The critical case}
\label{Subsection: the critical case}

In the critical case, we assume that the birth rate $\lambda$ and the death rate $\mu$ are equal 1. The general case can be reduced to this case using a time change.

Below is our main theorem in the critical case.  We allow the sampling time $T_n$ to depend on the population size.  The result shows that as long as $n/T_n \rightarrow 0$ as $n \rightarrow \infty$, the quantities $L^k_{n,T_n}$, normalized by the population size at the time of sampling, are asymptotically independent and normally distributed.

\begin{theorem}\label{Theorem}
In the critical case with $\lambda=\mu=1$, let $K$ be a fixed positive integer, and let $\{T_n\}$ be a sequence such that 
\begin{equation}\label{Introduction: sequence condition}
    \lim_{n\rightarrow\infty}\frac{n}{T_n}=0.
\end{equation} 
Then as $n\rightarrow\infty,$
$$
\sqrt{\frac{n}{\log n}}\left\{\left(\frac{L^k_{n,T_n}}{N_{T_n}}-\frac{1}{k}\right)\right\}_{k=1}^K \Rightarrow N(0, I_{K\times K})
$$
where $\Rightarrow$ denotes convergence in distribution as $n \rightarrow \infty$, $N(0, V)$ denotes the centered multivariate normal distribution with covariance matrix $V$, and $I_{K\times K}$ is the $K\times K$ identity matrix.
\end{theorem}

\begin{remark}
Recall that a famous result of Yaglom \cite{yaglom1947certain} states that for large $t$, the population size of a critical birth and death process at time $t$ conditioned on survival until time $t$ is approximately $Wt$, where $W$ is exponentially distributed with mean 1. Therefore, the assumption \eqref{Introduction: sequence condition} requires the population size at time $T_n$ to be much larger than the sample size.
\end{remark}

The quantity $L^k_{n,T_n}$ has been studied for Kingman's coalescent.  Fu and Li \cite{fu1993statistical} obtained $\E[L^1_{n,T_n}]=2$, and Fu \cite{fu1995statistical} obtained $\E[L^k_{n,T_n}]= 2/k$. Then Janson and Kersting \cite{janson2010external} established the asymptotic normality of $L^1_{n,T_n}$, and Dahmer and Kersting \cite{dahmer2015internal} showed that
\begin{equation}\label{dkresult}
\sqrt{\frac{n}{4\log n}}\left\{\left(L^k_{n}-\frac{2}{k}\right)\right\}_{k=1}^K \Rightarrow N(0,I_{K\times K}).
\end{equation}
Our result can be viewed as an analog of the result (\ref{dkresult}) for the critical birth and death process. The intuition behind this connection is that Kingman's coalescent models the genealogy of a sample from a population with fixed size. For the critical birth and death process, we will see that when $n$ is large, most of the coalescent events occur fairly quickly, when the population size is approximately $WT_n$ for some exponentially distributed random variable $W$. Therefore, if we scale the coalescent times by the population size, so that the randomness from $W$ is removed, then the genealogy of the sample should be well approximated by Kingman's coalescent.  Indeed, it is well known that critical branching processes converge to a limit process known as Feller's diffusion when the population size tends to infinity, and the genealogy of Feller's diffusion is a time change of Kingman's coalescent, an idea which goes back to \cite{perkins}; see also Theorem 5.1 of \cite{dk99}.

\subsection{The supercritical case}

In the supercritical case, the shape of the coalescent tree is much different from the shape in the critical case.  Because the population is growing rapidly, the population will be large a short time after time zero, and it is likely that most of the sampled individuals will be descended from distinct ancestors at that time.  Consequently, the coalescent tree is more star-shaped, with most of the coalescence occurring near the root of the tree.  Unlike in the critical case, there are long external branches supporting only one leaf.  The difference is illustrated in Figure \ref{fig:treeshape} below.

	\begin{figure}[h]
		\centering
		\begin{tikzpicture}[scale=0.75]
			\draw [very thick] (0,0)--(0.75,1);
			\draw [very thick] (1.5,0)--(0.75,1);
			\draw [very thick] (4.5,0)--(5.25,1.5);
			\draw [very thick] (6,0)--(5.25,1.5);
			\draw [very thick] (3,0)--(1.5,2);
			\draw [very thick] (0.75,1)--(1.5,2);
			\draw [very thick] (5.25,1.5)--(3.5,3.5);
			\draw [very thick] (1.5,2)--(3.5,3.5);
			
			\draw [very thick] (8.5,0)--(10.5, 2.5);
			\draw [very thick] (10,0)--(10.5, 2.5);
			\draw [very thick] (11.5,0)--(11,3);
			\draw [very thick] (13,0)--(12,3);
			\draw [very thick] (14.5,0)--(12,3);
			\draw [very thick] (10.5, 2.5)--(11,3);
			\draw [very thick] (12,3)--(11.5,3.5);
			\draw [very thick] (11,3)--(11.5,3.5);
		
			\node at (0,-0.3){1};
			\node at (1.5,-0.3){2};
			\node at (3,-0.3){3};
			\node at (4.5,-0.3){4};
			\node at (6,-0.3){5};
			
			\node at (8.5,-0.3){1};
			\node at (10,-0.3){2};
			\node at (11.5,-0.3){3};
			\node at (13,-0.3){4};
			\node at (14.5,-0.3){5};
			
			\node at (3.5, 4){Critical case};
			\node at (11.7, 4){Supercritical case};
		\end{tikzpicture}
  \caption{Comparison of the tree shapes in the critical and supercritical cases} \label{fig:treeshape}
	\end{figure}
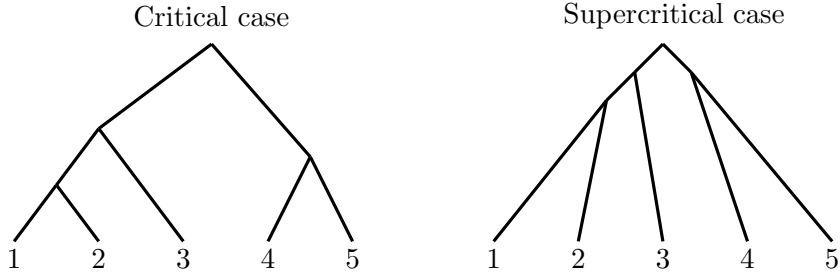

Durrett \cite{durrett2013population} considered the case in which the sample size $n$ and the growth rate $r > 0$ are fixed, and the sampling time $T$ tends to infinity.  He showed that when $2 \leq k \leq n-1$, 
we have 
\begin{equation}\label{durrlimit}
\lim_{T \rightarrow \infty} \E[L^k_{n,T}] = \frac{n}{rk(k-1)}.
\end{equation}
Similar calculations appear in \cite{williams2016identification} and \cite{Bozic2016clonal}, which focused on applications to cancer.
Gunnarsson, Leder, and Foo \cite{gunnarsson2021exact} calculated the exact expectation in the case when the entire population is sampled.

While the limit in expectation \eqref{durrlimit} holds for fixed $n$ as $T \rightarrow \infty$, it is also natural to consider the asymptotic distribution of $L_{n,T}^k$ when the sample size $n$ becomes large.
We allow $\lambda=\lambda_n$, $\mu=\mu_n$ and $T=T_n$ to depend on $n$. Allowing the birth and death rates to depend on $n$ makes it possible for the growth rate $r_n=\lambda_n-\mu_n$ to go to 0.
In \cite{johnson2023estimating}, we showed that if the condition \eqref{Introduction: sequence condition super LLN} below holds, then $$\frac{r_n}{n} L_{n,T_n}^1 - r_n T_n + \log n + 1 \Rightarrow \log W,$$
where $W$ has the exponential distribution with mean $1$.  We also showed in \cite{johnson2023estimating} that if $L_{n,T_n}^{in} = \sum_{k=2}^{n-1} L_{n,T_n}^k$ denotes the sum of the lengths of all internal branches in the coalescent tree, and if the stronger condition \eqref{Introduction: sequence condition super} below holds, then
$$\frac{r_n}{\sqrt{n}} \bigg( L_{n,T_n}^{in} - \frac{n}{r_n} \bigg) \Rightarrow N(0,1),$$
and $L_{n,T_n}^{in}$ and $L_{n,T_n}^1$ are asymptotically independent.  In \cite{johnson2023estimating}, we showed how to use these results to estimate the growth rate of a birth and death process, and applied the results to data on blood cancer.

Here we use the methods developed in \cite{johnson2023estimating} to obtain more precise asymptotics for $L_{n,T_n}^k$ when $k \geq 2$, sharpening the results in \cite{durrett2013population}.  Theorem \ref{Theorem: super LLN} gives a law of large numbers, and Theorem \ref{Theorem: super} establishes the asymptotic normality of $L_{n,T_n}^k$ under a stronger hypothesis.

\begin{theorem}\label{Theorem: super LLN}
Let $k\ge 2$ be a fixed integer. Assume that 
\begin{equation}\label{Introduction: sequence condition super LLN}
    \lim_{n\rightarrow\infty}n e^{-r_n T_n}=0.
\end{equation} 
Then as $n\rightarrow\infty,$ 
$$
\frac{r_n}{n}L_{n,T_n}^k\stackrel{P}{\longrightarrow} \frac{1}{k(k-1)},
$$
where $\stackrel{P}{\longrightarrow}$ denotes convergence in probability as $n \rightarrow \infty$.
\end{theorem}

\begin{theorem}\label{Theorem: super}
Let $K\ge 2$ be a fixed integer. Assume that 
\begin{equation}\label{Introduction: sequence condition super}
\lim_{n\rightarrow\infty}n^{3/2} (\log n) e^{-r_n T_n}=0.
\end{equation} 
Then as $n\rightarrow\infty,$ $$\left\{\frac{r_n}{\sqrt{n}}\left(L_{n,T_n}^k-\frac{n}{r_nk(k-1)}\right)\right\}_{k=2}^K \Rightarrow N(0,V),$$
where $V$ is a covariance matrix.  Moreover, if $2 \leq k' \leq k \leq K$, then
\begin{equation}\label{finalcov}
V_{k,k'} = \left\{
\begin{array}{ll}
\frac{2}{(k-1)^2(2k-1)} - \frac{2k+1}{k^2(k-1)^2} - \frac{2}{k-1} + \frac{\pi^2}{3}  - 2 \sum_{j=1}^{k-2} \frac{1}{j^2} & \mbox{ if } k = k' \\
\frac{2k-1}{(k-1)^2} - \frac{\pi^2}{3} + 2 \sum_{j=1}^{k-2} \frac{1}{j^2} & \mbox{ if } k = k' + 1 \\
- \frac{(2k-1)^2}{2k(k-1)(2k-3)} + \frac{\pi^2}{6} - \sum_{j=1}^{k-2} \frac{1}{j^2} & \mbox{ if } k = k' + 2 \\
\frac{2}{(k-1)(k'-1)(k+k'-1)} - \frac{k+k'+1}{k(k-1)k'(k'-1)} - \frac{1}{k'(k-k'-2)} \\
\qquad \qquad+ \frac{1}{(k'-1)(k-k')} + \frac{2}{(k-k'-2)(k-k'-1)(k-k')} \sum_{j=k'}^{k-2} \frac{1}{j} & \mbox{ if }k \geq k'+3.
\end{array}\right.
\end{equation}
\end{theorem}

\begin{remark}\label{Remark: sequence condition super}
Since the expected population size is $e^{r_n T_n}$, equation \eqref{Introduction: sequence condition super LLN} requires the sample size to be much smaller than the population size.  Equation \eqref{Introduction: sequence condition super} is a somewhat stronger condition.
Note that equation \eqref{Introduction: sequence condition super LLN} implies
$$\lim_{n\rightarrow\infty} r_n T_n =\infty.$$
\end{remark}

\begin{remark}
Although we have proved Theorems \ref{Theorem: super LLN} and \ref{Theorem: super} when the birth and death rates do not change over time, the results may hold in greater generality.  Cheek \cite{cheek2022coalescent} showed that the asymptotic formula for the coalescence times as $T \rightarrow \infty$ is robust in the sense that, even if the birth rate slows down as the population size approaches a carrying capacity, we obtain the same limit if we take the sampling time and the carrying capacity to infinity appropriately. The idea is that because the genealogical tree is star-shaped, what should matter for the results is that the population is growing exponentially during the time period corresponding to when most of the mergers of the ancestral lines are taking place.
\end{remark}

\subsection{Organization of the paper}
In Section \ref{Section: Coalescent point process}, we introduce the coalescent point process, which will be the main tool that we use to analyze the branch lengths. We will give a detailed proof of our main result in the critical case in Sections~\ref{Section: approximation of H_i,n,T}, \ref{Section: central limit theorem for an approximation of L^k_i,n,T/N_T} and \ref{Section: error bounds}.  In Section \ref{Section: approximation of H_i,n,T}, we will derive an approximation for the branch lengths.  In Section \ref{Section: central limit theorem for an approximation of L^k_i,n,T/N_T}, we prove a central limit theorem using this approximation. We bound the error from the approximation in Section \ref{Section: error bounds}. The proofs of Theorems \ref{Theorem: super LLN} and Theorem \ref{Theorem: super} in the supercritical case are analogous and can be found in Section \ref{Section: proof super}.  In the supercritical case, the necessary approximations and error bounds were established in the supplemental information in \cite{johnson2023estimating}.

\section{Coalescent point process}\label{Section: Coalescent point process}
\subsection{Introduction to the coalescent point process}

The main tool that we apply to study the site frequency spectrum is the coalescent point process (CPP). A CPP is a point process which can be used to represent the genealogical tree in certain population models, including birth and death processes.  The CPP can be constructed from a sequence of i.i.d. random variables $H_1, H_2,\dots$, which represent coalescence times.  The genealogical tree is then obtained as follows.  We first draw a vertical line of height $T$.  We then draw vertical lines of heights $H_1,H_2,\dots$, and at the top of each of these vertical lines, we draw a horizontal line to the left, stopping when it hits a vertical branch.  We continue this process until one of the random variables $H_i$ is larger than $T$, and then we stop, which means that if $H_i < T$ for $i = 1, 2, \dots, n-1$ and $H_n > T$, then the population will consist of $n$ individuals labelled $0, 1, \dots, n-1$ (see Figure \ref{Figure: p=1}).  In particular, for populations whose genealogy can be represented in this way, the population size at time $T$ must have a geometric distribution.

\begin{figure}[h]
\centering
\begin{tikzpicture}[scale=0.75]
\draw(0,0)--(0,7);
\draw(0,0)--(18,0);

\node at (0,-0.5){0};
\node at (2,-0.5){1};
\node at (4,-0.5){2};
\node at (6,-0.5){3};
\node at (8,-0.5){4};
\node at (10,-0.5){5};
\node at (12,-0.5){6};
\node at (14,-0.5){7};
\node at (16,-0.5){8};
\node at (-0.5,6.5){$T$};
\fill (0,6.5) circle(1.5pt);

\draw[very thick](0,0)--(0,6.5);
\draw[very thick](2,0)--(2,3.5);
\draw[very thick](2,3.5)--(2,5.5);
\node at (2.5,5.5){$H_1$};
\draw[very thick](4,0)--(4,3.5);
\node at (4.5,3.5){$H_2$};
\draw[very thick](6,0)--(6,1.5);
\draw[very thick](6,1.5)--(6,4.5);
\node at (6.5,4.5){$H_3$};
\draw[very thick](8,0)--(8,1.5);
\node at (8.5,1.5){$H_4$};
\draw[very thick](10,0)--(10,3);
\draw[very thick](10,3)--(10,6);
\node at (10.5,6){$H_5$};
\draw[very thick](12,0)--(12,1);
\draw[very thick](12,1)--(12,3);
\node at (12.5,3){$H_6$};
\draw[very thick](14,0)--(14,1);
\node at (14.5,1){$H_7$};
\draw[very thick](16,0)--(16,4);
\node at (16.5,4){$H_8$};

\draw[dashed] (0,5.5)--(2,5.5);
\draw[dashed] (2,3.5)--(4,3.5);
\draw[dashed] (2,4.5)--(6,4.5);
\draw[dashed] (6,1.5)--(8,1.5);
\draw[dashed] (0,6)--(10,6);
\draw[dashed] (10,3)--(12,3);
\draw[dashed] (12,1)--(14,1);
\draw[dashed] (10,4)--(16,4);
\end{tikzpicture}
\caption{The genealogical tree with coalescence times $H_1,H_2,\dots,H_8$, ($H_9>T$).} \label{Figure: p=1}
\end{figure}
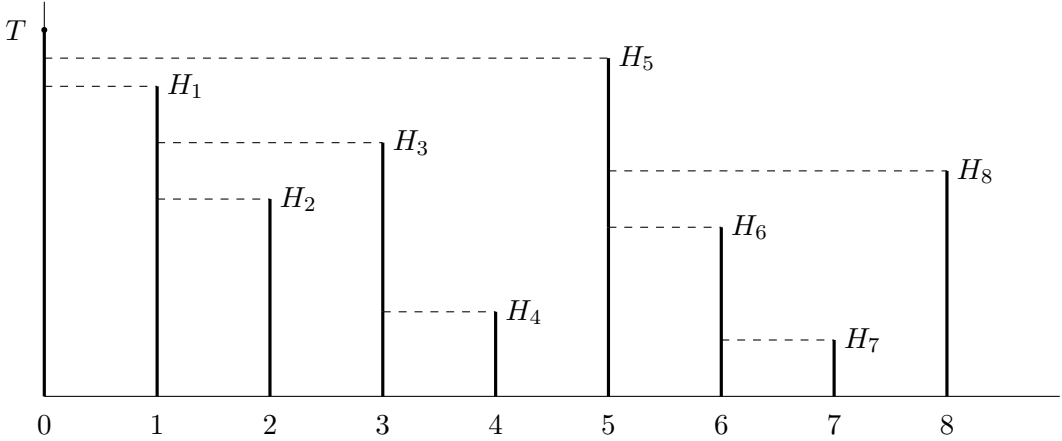

The coalescent point process was introduced by Popovic in \cite{popovic2004asymptotic} to study the critical birth and death tree conditioned to have $n$ leaves at time $T$.  Using the same technique, Aldous and Popovic \cite{aldous2005critical} studied other quantities such as the time to the most common recent ancestor. Asymptotic formulas for the coalescence times when $n$ and $T$ are large were also obtained in \cite{popovic2004asymptotic} using the excursion theory for Brownian motion.  In this limit, the population size becomes infinite, so the CPP is represented not by a sequence of i.i.d. random variables $H_i$ but by the points of a Poisson point process.  The corresponding CPP is known as the Brownian CPP and represents the scaling limit of a critical branching process conditioned to survive for a long time. The Brownian CPP also arises as the scaling limit of the genealogy of a sample from a population with constant size, conditional on the entire sample belonging to a small subpopulation \cite{lambert2019recovering}.

\subsection{Sampling from the coalescent point process}

Since one does not usually have data on the entire population, there has been extensive study of how to describe the genealogy of a sample. Two ways of sampling have been studied in depth: sampling each individual independently with probability $p$ (Bernoulli sampling) or taking a sample of size $n$, uniformly from the population (uniform sampling). Stadler \cite{stadler2009incomplete}, considered a Bernoulli sample from a birth and death process, conditional on sampling $n$ individuals. She gave an explicit formula for the joint density of the coalescent times when a uniform prior is assumed for the time from the origin. In \cite{lambert2013birth}, Lambert and Stadler simplified the formula in the Bernoulli sampling scheme and also considered a more general setting where the death rate could be age-dependent. They made the following observations:
\begin{enumerate}
\item The number of unsampled individuals between sampled individuals has a geometric distribution with success probability $p$.
\item If $i < j$ and individuals $i$ and $j$ are sampled but individuals $i+1, \dots, j-1$ are not, then the coalescence time between individual $i$ and $j$ is $\max\{H_{i+1},\dots,H_j\}$.
\end{enumerate}
Using these two observations, they showed that a Bernoulli sample of a CPP is another CPP in which the random variables $H_i$ representing the coalescence times have the same distribution as the maximum of a geometric($p$) number of the corresponding random variables in the original CPP.  The resulting formula for the distribution of coalescence times can be found in Proposition 5 of \cite{lambert2013birth}. 
In the uniform sampling scheme, Stadler found the expectation of the $k$th coalescence time in a birth and death process in \cite{stadler2008lineages}.
Then Lambert \cite{lambert2018coalescent} showed that a uniform sample of size $n$ can be realized as a mixture of Bernoulli sampling, conditional on sampling $n$ individuals. Based on this representation, he obtained an explicit formula for the joint density of the coalescence times for birth and death process.  We now present this formula and explain how it is derived in \cite{lambert2018coalescent}.


We start by considering the genealogy of the whole population. The density of the coalescence times for a birth and death process with birth rate $\lambda$ and death rate $\mu$, conditioned to survive until time $T$, is given by equations (4) and (5) in \cite{lambert2018coalescent}:
\begin{equation}\label{Hdensity}
f_{H}(t)=\left\{\begin{aligned}
&\frac{\lambda r^2 e^{-rt}}{(\lambda-\mu e^{-rt})^2}\mathbbm{1}_{\{t> 0\}}\qquad&\mbox{if } \lambda\neq \mu,\\
&\frac{\lambda}{(1+\lambda t)^2}\mathbbm{1}_{\{t> 0\}} \qquad&\mbox{if } \lambda=\mu.
\end{aligned}
\right.
\end{equation}
If instead we sample each individual independently with probability $y$, then the new density of coalescence times is given in equation (7) in \cite{lambert2018coalescent} to be
\begin{equation*}
f_{H_y}(t)=\left\{
\begin{aligned}
&\frac{y\lambda r^2 e^{-rt}}{(y\lambda+(r-y\lambda)e^{-rt})^2}\mathbbm{1}_{\{t> 0\}}\qquad&\mbox{if }\lambda\neq\mu,\\
&\frac{y\lambda}{(1+y\lambda t)^2}\mathbbm{1}_{\{t> 0\}}\qquad&\mbox{if }\lambda=\mu.\\
\end{aligned}
\right.
\end{equation*}
This density is calculated by finding the density of the maximum of a geometrically distributed number of random variables whose density is given by \eqref{Hdensity}.

To construct the genealogical tree of a uniform sample of size $n$, we can use the following two-step process.
First, one can choose the sampling probability $y\in(0,1)$ from the density 
$$
f_{Y_{n,T}}(y)=\frac{n\delta_T y^{n-1}}{(y+\delta_T-y\delta_T)^{n+1}} \mathbbm{1}_{\{0 < y < 1\}}.
$$ 
This density appears as equation (12) in \cite{lambert2018coalescent}; note that $\delta_T$ here is $1-a$ in \cite{lambert2018coalescent}. Then conditional on $Y_{n,T} = y$, we choose the coalescence times $H_{i,n,T}$ for $i = 1, 2, \dots, n-1$ to have the same distribution as the conditional distribution of $H_y$ given $H_y \leq T$, which has the effect of conditioning the sample size to equal $n$.  We get
\begin{equation*}
f_{H_{i,n,T}|Y_{n,T}=y}(t)=\left\{
\begin{aligned}
&\frac{y\lambda+(r-y\lambda)e^{-rT}}{y\lambda(1-e^{-rT})} \cdot \frac{y\lambda r^2 e^{-rt}}{(y\lambda+(r-y\lambda)e^{-rt})^2}\mathbbm{1}_{\{0< t<T\}},\qquad&\mbox{if }\lambda\neq\mu,\\
&\frac{1+y\lambda T}{y\lambda T} \cdot \frac{y\lambda}{(1+y\lambda t)^2}\mathbbm{1}_{\{0<t< T\}},\qquad&\mbox{if }\lambda=\mu.\\
\end{aligned}
\right.
\end{equation*}
The joint density of the coalescence times is therefore
$$
f_{H_{1,n,T},\dots, H_{n-1,n,T}}(t_1\dots,t_{n-1}) = \int_0^1 \left(\prod_{i=1}^{n-1} f_{H_{i,n,T}|Y_{n,T}=y} (t_i)\right) f_{Y_{n,T}}(y)\ dy.
$$
The same joint density of the coalescence times appeared in Proposition 19 of \cite{harris2020coalescent}.

Suppose $Y_{n,T} = y$, so that we are sampling each individual with probability $y$ in the second step of this construction.  To perform the sampling, we can examine the individuals labelled $0, 1, 2, \dots$ in sequence, selecting each one with probability $y$.  After $i$ sampled individuals have been found, we let $X_{i,n,T}$ denote the number of individuals that we need to examine to find the next sampled individual.  We may instead reach the end of the population, in which case if there are $k$ individuals in the population after the $i$th sampled one, then we set $X_{i,n,T} = k+1$. This means that when we condition the number of sampled individuals to equal $n$, for $i = 1, \dots, n$, the label of the $i$th sampled individual is $(\sum_{j=0}^{i-1} X_{j,n,T}) - 1$, and the population size is
\begin{equation}\label{Sampling from the coalescent point process: population size}    
N_T=\sum_{i=0}^{n} X_{i,n,T}-1.
\end{equation}
Conditional on $X_{i,n,T}\ge m$, the event $X_{i,n,T}\ge m+1$ occurs if and only if the next coalescence time is less than $T$ (which means that we do not reach the end of the population) and the next individual is not sampled. Therefore,
\begin{equation}\label{Sampling from the coalescent point process: X_i,n,T p.m.f}
\mathbb{P}(X_{i,n,T}\ge m+1|X_{i,n,T}\ge m) = (1-\delta_T)(1-y), 
\end{equation}
where
\begin{equation}\label{Sampling from the coalescent point process: definition of delta_T}
\delta_T=\mathbb{P}(H\ge T)=\left\{
\begin{aligned}
&\frac{re^{-rT}}{\lambda(1-e^{-rT})+re^{-rT}}\qquad&\lambda\neq\mu,\\
&\frac{1}{1+\lambda T}\qquad &\lambda=\mu.
\end{aligned}
\right.
\end{equation}
That is, the random variables $X_{i,n,T}$ are independent geometric random variables with probability mass function $$\mathbb{P}(X_{i,n,T}=k)=(1-\delta_T)^{k-1}(1-y)^{k-1}(1-(1-\delta_T)(1-y))\qquad\text{for } k=1,2,\dots.$$

\subsection{Recovering the branch lengths from the coalescent point process}\label{Subsection: recoveing the branch length from coalescent point process}

Given the coalescence times $\{H_{i,n,T}\}_{i=1}^{n-1}$ of a CPP, one can construct the genealogy as in Figure~\ref{Figure: recovering the branch length}. The total length $L_{n,T}^k$ of the branches supporting $k$ leaves is a function of these coalescence times. To be more precise, under the convention that the maximum over an empty set is 0, the length of the portion of the 0th branch that supports $k$ leaves is (see Figure \ref{Figure: recovering the branch length} for an illustration)
\begin{equation}\label{Recovering the branch lenght from coalescent point process: definition of L^k_0,n,T}
\begin{split}
L^{k}_{0,n,T}&=\left(H_{k,n,T}-\max_{1\le i\le k-1}H_{i,n,T}\right)^{+},
\end{split}
\end{equation}
the length of the portion of the $i$th branch that supports $k$ leaves is
\begin{equation}\label{Recovering the branch lenght from coalescent point process: definition of L^k_i,n,T}
L^{k}_{i,n,T}=\left(H_{i,n,T}\wedge H_{i+k,n,T}-\max_{i+1\le j\le i+k-1} H_{j,n,T}\right)^{+}\qquad \text{ for }1\le i\le n-k-1,
\end{equation}
the length of the portion of the $(n-k)$th branch that supports $k$ leaves is 
\begin{equation}\label{Recovering the branch lenght from coalescent point process: definition of L^k_n-k,n,T}
L^{k}_{n-k,n,T}=\left(H_{n-k,n,T}-\max_{n-k+1\le j\le n-1} H_{j,n,T}\right)^{+},
\end{equation}
and there is no portion of the $i$th branch that supports $k$ leaves for $n-k+1\le i\le n-1$. Therefore, the total length of branches supporting $k$ leaves is 
\begin{equation}\label{Recovering the branch length from the coalescent point process: definition of L^k_n,T}
L^{k}_{n,T}=\sum_{i=0}^{n-k}L^{k}_{i,n,T}. 
\end{equation}
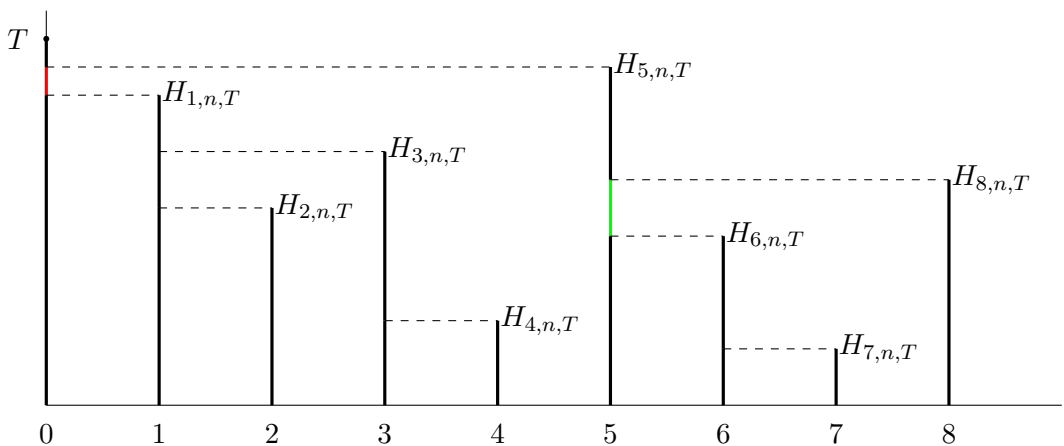
\begin{figure}[ht]
\centering
\begin{tikzpicture}[scale=0.75]
\draw(0,0)--(0,7);
\draw(0,0)--(18,0);

\node at (0,-0.5){0};
\node at (2,-0.5){1};
\node at (4,-0.5){2};
\node at (6,-0.5){3};
\node at (8,-0.5){4};
\node at (10,-0.5){5};
\node at (12,-0.5){6};
\node at (14,-0.5){7};
\node at (16,-0.5){8};
\node at (-0.5,6.5){$T$};
\fill (0,6.5) circle(1.5pt);

\draw[very thick](0,0)--(0,5.5);
\draw[very thick, red](0,5.5)--(0,6);
\draw[very thick](0,6)--(0,6.5);
\draw[very thick](2,0)--(2,3.5);
\draw[very thick](2,3.5)--(2,5.5);
\node at (2.75,5.5){$H_{1,n,T}$};
\draw[very thick](4,0)--(4,3.5);
\node at (4.75,3.5){$H_{2,n,T}$};
\draw[very thick](6,0)--(6,1.5);
\draw[very thick](6,1.5)--(6,4.5);
\node at (6.75,4.5){$H_{3,n,T}$};
\draw[very thick](8,0)--(8,1.5);
\node at (8.75,1.5){$H_{4,n,T}$};
\draw[very thick](10,0)--(10,3);
\draw[very thick, green](10,3)--(10,4);
\draw[very thick](10,4)--(10,6);
\node at (10.75,6){$H_{5,n,T}$};
\draw[very thick](12,0)--(12,1);
\draw[very thick](12,1)--(12,3);
\node at (12.75,3){$H_{6,n,T}$};
\draw[very thick](14,0)--(14,1);
\node at (14.75,1){$H_{7,n,T}$};
\draw[very thick](16,0)--(16,4);
\node at (16.75,4){$H_{8,n,T}$};

\draw[dashed] (0,5.5)--(2,5.5);
\draw[dashed] (2,3.5)--(4,3.5);
\draw[dashed] (2,4.5)--(6,4.5);
\draw[dashed] (6,1.5)--(8,1.5);
\draw[dashed] (0,6)--(10,6);
\draw[dashed] (10,3)--(12,3);
\draw[dashed] (12,1)--(14,1);
\draw[dashed] (10,4)--(16,4);
\end{tikzpicture}
\caption{The genealogical tree with $H_{1,n,T},\dots,H_{8,n,T}$. The red segment is the portion of the 0th branch supporting 5 leaves. The green segment is the portion of the 5th branch supporting 3 leaves.} 
\label{Figure: recovering the branch length}
\end{figure}

This representation allows one to study the site frequency spectrum. Lambert \cite{lambert2008allelic} obtained the expected site frequency spectrum in a Bernoulli $y$-sample from a general CPP, conditional to have size $n$. In the special case of supercritical birth and death process, Dinh et. al. \cite{dinh2020statistical} obtained an explicit formula. Ignatieva, Hein, and Jenkins \cite{ignatieva2020characterisation} took the limit as $y\rightarrow0$ in this formula and obtained the same $\frac{1}{k(k-1)}$ shape for the expected site frequency spectrum. One can also apply similar ideas to use the CPP to study the closely related allele frequency spectrum, in which one partitions the sampled individuals into blocks consisting of individuals that acquired exactly the same mutations and counts the number of blocks of size $k$.  Champagnat and Lambert studied the allele frequency spectrum in a birth and death process with age dependent death rates. They derived the expected allele frequency spectrum for the genealogy of the entire population in \cite{champagnat2012splitting}, and they studied the size of the largest family and age of the oldest family in \cite{champagnat2013splitting}. Champagnat and Henry \cite{champagnat2016moments} studied the factorial moments for the allele frequency spectrum.  

The site frequency spectrum is closely linked to the branch lengths $L_{n,T}^k$ when mutations occur at a constant rate along the branches.  However, one can also consider a model in which mutations can appear only when a new individual is born. Richard \cite{richard2014splitting} considered a birth and death process with age-dependent death rates and mutations at birth times. He used the CPP to give the expected allele frequency spectrum and prove almost sure convergence for the allele frequency spectrum under proper scaling. In \cite{delaporte2016mutational}, Delaporte, Achaz, and Lambert considered a critical birth and death process with mutations at birth times. They sampled each individual independently with probability $p$, and conditioned the sample to have size $n$. They gave the expected site frequency spectrum both when the sampling time $T$ is fixed, and when this time is unknown and is given an improper prior distribution.

\section{Approximation of $H_{i,n,T}$}\label{Section: approximation of H_i,n,T}

In this section, we consider the critical case with birth and death rates equal to 1. We derive an approximation of $H_{i,n,T}$ when $n$ and $T$ are large. The corresponding approximation of $H_{i,n,T}$ in the supercritical case, which we review in Section \ref{Section: proof super}, is based on the same idea and was derived in the supplemental information in \cite{johnson2023estimating}.

\subsection{Approximation of $H_{i,n,T}$ for large $T$}\label{Subsection: approximation of H_i,n,T for large T}
As in Section 1.1 of the supplemental information in \cite{johnson2023estimating}, we define $\delta_T$ as in \eqref{Sampling from the coalescent point process: definition of delta_T} and then write $Y_{n,T}=\delta_T Q_{n,T}$ where $Q_{n,T}$ has density
\begin{equation}\label{Approximation of H_i,n,T for large T: f_Q_n,T}
f_{Q_{n,T}}(q)=\frac{n q^{n-1}}{(1+q-q\delta_T)^{n+1}},\qquad q\in(0,1/\delta_T).
\end{equation}
Replacing $y$ by $q\delta_T$, the conditional density for $H_{i,n,T}$  given $Q_{n,T}=q$ is
$$
f_{H_{i,n,T}|Q_{n,T}=q}(t)=\frac{1+q\delta_T T}{T}\frac{1}{(1+q\delta_Tt)^2}, \qquad t\in(0,T).
$$
By writing $U_{i,n,T}=Q_{n,T}H_{i,n,T}/T$, the conditional density for $U_{i,n,T}$ given $Q_{n,T}=q$ is
\begin{equation}\label{Approximation of H_i,n,T for large T: f_U_i,n,T|Q_n,T=q}
f_{U_{i,n,T}|Q_{n,T}=q}(u)=\frac{1+q\delta_T T}{q}\frac{1}{(1+\delta_T Tu)^2}, \qquad u\in(0,q).
\end{equation}
Therefore, to obtain $\{H_{i,n,T}\}_{i=1}^{n-1}$, we can:
\begin{enumerate}
\item Choose $Q_{n,T}$ from the density $f_{Q_{n,T}}$ in \eqref{Approximation of H_i,n,T for large T: f_Q_n,T}.
\item Given $Q_{n,T}=q,$ choose $U_{i,n,T}$ for $i=1,\dots, n-1$ independently from the conditional density $f_{U_{i,n,T}|Q_{n,T}=q}$ in \eqref{Approximation of H_i,n,T for large T: f_U_i,n,T|Q_n,T=q}.
\item Let $H_{i, n,T}=TU_{i,n,T}/Q_{n,T}$ for $i=1,\dots n-1$.
\end{enumerate} 

As $T\rightarrow\infty$, we have $\delta_T\rightarrow0$ and therefore the density $f_{Q_{n,T}}$ converges pointwise to
\begin{equation}\label{Approximation of H_i,n,T for large T: f_Q_n}
f_{Q_{n}}(q)=\frac{n q^{n-1}}{(1+q)^{n+1}},\qquad q\in(0,\infty).   
\end{equation}
As $T\rightarrow\infty$ we have $\delta_T T\rightarrow1$, and therefore the conditional density $f_{U_{i,n,T}|Q_{n,T}=q}$ converges pointwise to
\begin{equation}\label{Approximation of H_i,n,T for large T: f_U_i,n|Q_n=q}
f_{U_{i,n}|Q_{n}=q}(u)=\frac{1+q}{q}\frac{1}{(1+u)^2}, \qquad u\in(0,q).
\end{equation}
Therefore, we can obtain an approximation $H_{i,n}$ to $H_{i,n,T}$ which is valid for large $T$ in the following way:
\begin{enumerate}
\item Choose $Q_{n}$ from the density $f_{Q_{n}}$ in \eqref{Approximation of H_i,n,T for large T: f_Q_n}.
\item Given $Q_{n}=q,$ choose $U_{i,n}$  for $i=1,\dots,n-1$ independently from the conditional density $f_{{U_{i,n}}|Q_{n}=q}$ in \eqref{Approximation of H_i,n,T for large T: f_U_i,n|Q_n=q}.
\item Let ${H}_{i,n}=TU_{i,n}/Q_{n}$ for $i=1,\dots, n-1$.
\end{enumerate}

\begin{remark}
The resulting joint density for the coalescence times also arises as the scaling limit of coalescence times in more general settings. Harris, Johnston and Roberts \cite{harris2020coalescent} obtained this joint density in the limit as $T \rightarrow \infty$ for critical Galton-Watson processes with a general offspring distribution having finite variance; see the formula for $f_k(s_1,\dots, s_{k-1})$ in Section 2.3 of \cite{harris2020coalescent}.
\end{remark}

\subsection{Approximation of $H_{i,n,T}$ for large $n$ and $T$}	
\label{Subsection: approximation of H i,n,T for large n and T}

As in Section 1.1 of the supplemental information to \cite{johnson2023estimating}, we scale $Q_{n}$ by $n$. More precisely, the density for  $Q_{n}/n$ is
$$
f_{Q_n/n}(q)=\frac{q^{n-1}}{\left(\frac{1}{n}+q\right)^{n+1}},\qquad q\in(0,\infty).
$$
Replacing $q$ by $nq$, the conditional density of $U_{i,n}$ given $Q_n/n=q$ is
\begin{equation}\label{Approximation of H_i,n,T for large n and T: f_U_i,n|Q_n/n=q}
f_{{U_{i,n}}|Q_n/n=q}(u)=\frac{1+nq}{nq}\frac{1}{(1+u)^2},\qquad q\in(0,nq).
\end{equation}

As $n\rightarrow\infty$, the density $f_{Q_n/n}(q)$ converges pointwise to
$$
f_{1/W}(q)=\lim _{n \rightarrow \infty} \frac{q^{n-1}}{\left(\frac{1}{n}+q\right)^{n+1}}=\frac{e^{-1 / q}}{q^{2}}, \qquad q\in(0,\infty),
$$
which is the density function for $1 / W$ where $W$ has an exponential distribution with mean 1. 
For each fixed $q>0$, as $n\rightarrow\infty$, the conditional density $f_{U_{i,n}|Q_n/n=q}$ converges pointwise to 
\begin{equation}\label{Approximation of H_i,n,T for large n and T: f_U_i}
 f_{U_i}(u)=\frac{1}{(1+u)^2},\qquad u\in(0,\infty). 
\end{equation}
Therefore, we can obtain an approximation $H_i$ to $H_{i,n}$ which is valid for large $n$ in the following way:
\begin{enumerate}
\item Choose $W$ from the exponential density
$$f_{W}(q)=e^{-q},\qquad q\in(0,\infty).$$
\item Choose $U_i$ for $i=1,\dots,n-1$ independently from the density $f_{U_i}$ in \eqref{Approximation of H_i,n,T for large n and T: f_U_i}.
\item Let ${H}_{i}=TWU_i/n$ for $i=1,\dots, n-1$.
\end{enumerate}

\begin{remark}
To understand the scaling in the formula for $H_i$, note that near time $T$, the size of the population is approximately $TW$. Also, we expect the rate at which two lineages merge to be approximately the reciprocal of the population size.  Therefore, as long as the number of lineages is comparable to $n$, the amount of time that it takes before a given lineage merges with one of the other lineages should be on the scale of $TW/n$, as we see in the formula for $H_i$.
\end{remark}


\section{Central limit theorem for an approximation of $L^{k}_{i,n,T}/N_T$}\label{Section: central limit theorem for an approximation of L^k_i,n,T/N_T}

Throughout the rest of the paper, $C$ will be some constant depending only on $K$ which may vary from line to line. For real numbers $a_n, b_n, c_n$ and random variables $X_n, Y_n, Z_n$, we use $o(\cdot)$, $O(\cdot)$, $o_p(\cdot)$, $O_p(\cdot)$ defined below:
\begin{equation*}
\begin{split}
a_n = o(b_n)  &\iff  a_n=b_nc_n,\  \lim_{n\rightarrow\infty}c_n =0,\\
a_n = O(b_n)  &\iff  a_n=b_nc_n,\  \sup|c_n| < \infty,\\
X_n = o_p(Y_n)  &\iff  X_n=Y_nZ_n,\  Z_n\text{\ converges to 0 in probability},\\
X_n = O_p(Y_n)  &\iff  X_n=Y_nZ_n,\  \{Z_n\} \text{ is tight}.
\end{split}
\end{equation*}
Similarly, for quantities indexed by $n$ and $T$, we write
\begin{equation*}
\begin{split}
a_{n,T} = O(b_{n,T})  &\iff  a_{n,T}=b_{n,T}c_{n,T},\  \sup|c_{n,T}| < \infty,\\
X_{n,T} = O_p(Y_{n,T})  &\iff  X_{n,T}=Y_{n,T}Z_{n,T},\  \{Z_{n,T}\} \text{ is tight}.
\end{split}
\end{equation*}

As we saw in Section \ref{Subsection: the critical case}, if we scale the branch lengths by the population size, we expect the genealogy to be similar to Kingman's coalescent. Therefore, we consider the rescaled branch lengths $L^{k}_{i,n,T}/N_T$. Recall from Section \ref{Subsection: recoveing the branch length from coalescent point process} that the branch lengths $L^{k}_{i,n,T}$ can be recovered from the coalescent times $H_{i,n,T}$. We use the approximations in Section \ref{Section: approximation of H_i,n,T} to obtain approximations of $L^{k}_{i,n,T}/N_T$ to be defined in the next paragraph. For the approximations, a truncation is needed to ensure that the second moment is finite. In the study of the external branch lengths of Kingman's coalescent, Janson and Kersting \cite{janson2010external} truncated the branches when there are $O(\sqrt{n})$ remaining lineages, so the time is of order $O(1/\sqrt{n})$. We use a similar approach and truncate the times $U_i/n$ at $(\log\log n)/\sqrt{n}$.

Recall from the definition of $L^k_{i,n,T}$ in \eqref{Recovering the branch lenght from coalescent point process: definition of L^k_i,n,T} and the construction of $H_{i,n,T}$ in Section \ref{Subsection: approximation of H_i,n,T for large T} that the rescaled branch length $L^k_{i,n,T}/N_T$ is
\begin{equation}\label{Central limit theorem for an approximation: formula  L^k_i,n,T/N_T of L^k_i,n,T/N_T}
\frac{L^k_{i,n,T}}{N_T}=\frac{T}{Q_{n,T}N_T}\left(U_{i,n,T}\wedge U_{i+k,n,T}-U_{i+1,n,T}\vee\dots\vee U_{i+k-1,n,T}\right)^+,\quad 1\le i\le n-k-1.   
\end{equation}
We define its approximation with truncation when $n$ and $T$ are large
\begin{equation}\label{Central limit theorem for an approximation: definition of widetilde L^k_i}
\widetilde{L}^k_{i}=\frac{1}{n}(U_{i}\wedge U_{i+k}\wedge \sqrt{n}\log\log n-U_{i+1}\vee \dots\vee U_{i+k-1})^+,\quad 1\le i\le n-k-1,
\end{equation}
when $T$ is large
\begin{equation}\label{Central limit theorem for an approximation: definition of widetilde L^k_i,n}
\widetilde{L}^{k}_{i,n}=\frac{1}{n}(U_{i,n}\wedge U_{i+k,n}\wedge\sqrt{n}\log\log n-U_{i+1,n}\vee \dots\vee U_{i+k-1,n})^+,\quad 1\le i\le n-k-1,  
\end{equation}
and for fixed $n$ and $T$
\begin{equation}\label{Central limit theorem for an approximation: definition of widetilde L^k_i,n,T}
\widetilde{L}^{k}_{i,n,T}=\frac{1}{n}(U_{i,n,T}\wedge U_{i+k,n,T}\wedge\sqrt{n}\log\log n-U_{i+1,n,T}\vee \dots\vee U_{i+k-1,n,T})^+,\quad 1\le i\le n-k-1.
\end{equation}
\begin{remark}
Note that
\begin{equation*}
\begin{split}
n\widetilde{L}^k_{i}&=(U_{i}\wedge U_{i+k}\wedge \sqrt{n}\log\log n-U_{i+1}\vee\dots\vee U_{i+k-1})^+\\
&=((U_{i}\wedge \sqrt{n}\log\log n)\wedge (U_{i+k}\wedge \sqrt{n}\log\log n)\\
&\qquad \qquad-(U_{i+1}\wedge \sqrt{n}\log\log n)\vee\dots\vee (U_{i+k-1}\wedge \sqrt{n}\log\log n))^+,
\end{split}
\end{equation*}
so we have actually truncated all the random variables $U_i$, and likewise for $U_{i,n}$ and $U_{i,n,T}$.
\end{remark}

We compute the moments involving the random variables $U_i$ in Section \ref{Subsection: moment computation} and then prove a  central limit theorem involving the $\widetilde{L}^{k}_i$ in Section \ref{Subsection: central limit theorem for an approximated version}.

\subsection{Moment computations}\label{Subsection: moment computation}
\begin{lemma}\label{Lemma: moment computation}
Let $\widetilde{L}^{k}_{i}$ be defined as in \eqref{Central limit theorem for an approximation: definition of widetilde L^k_i}. Then
\begin{align*}
\E\left[n\widetilde{L}^k_{i}\right]&=\frac{1}{k}+O\left(\frac{1}{\sqrt{n}\log\log n}\right), \\
\Var\left(n\widetilde{L}^k_{i}\right)&=\log n +2\log\log\log n+O(1), \\
\Cov(n\widetilde{L}^k_{i},n\widetilde{L}^{k'}_{i'})&=O(1),\qquad\textup{ for } (i,k)\neq(i',k').
\end{align*}
\end{lemma}

\begin{proof}
In this proof, we use that 
$$\mathbb{P}(U_{i}\ge x)=\int_{x}^\infty \frac{1}{(1+u)^2}\ du=\frac{1}{1+x}.$$
For the mean, using the inequality 
\begin{equation}\label{Central limit theorem for an approximated version: analysis inequality}
1-\frac{k}{x+1}\le \left(\frac{x}{1+x}\right)^{k}\le 1,\qquad\text{ for } k\in\mathbb{N}, \: x\ge 0   
\end{equation}
in the last line, we have
\begin{equation*}
\begin{split}
&\E\left[(U_{i}\wedge U_{i+k}\wedge \sqrt{n}\log\log n-U_{i+1}\vee\dots\vee U_{i+k-1})^+\right]\\
&\qquad=\int_0^\infty\mathbb{P}(U_{i+1}\vee\dots\vee U_{i+k-1}\le x\le U_{i}\wedge U_{i+k} \wedge \sqrt{n}\log\log n)\ dx\\
&\qquad=\int_0^{\sqrt{n}\log\log n}\mathbb{P}(U_{i+1}\vee\dots\vee U_{i+k-1}\le x\le U_{i}\wedge U_{i+k} )\ dx\\
&\qquad= \int_0^{\sqrt{n}\log\log n} \left(\frac{x}{x+1}\right)^{k-1}\frac{1}{(1+x)^2} \ dx\\
&\qquad=\frac{1}{k}\left(\frac{\sqrt{n}\log\log n}{\sqrt{n}\log\log n+1}\right)^{k}\\
&\qquad=\frac{1}{k}+O\left(\frac{1}{\sqrt{n}\log\log n}\right).
\end{split}
\end{equation*}
    
For the variance, we have
\begin{align*}
&\E\left[\left((U_{i}\wedge U_{i+k}\wedge \sqrt{n}\log\log n-U_{i+1}\vee\dots\vee U_{i+k-1})^+\right)^2\right]\\
&\qquad=\int_0^{\infty}\int_0^\infty \mathbb{P}(U_{i+1}\vee\dots\vee U_{i+k-1}\le x\le  U_{i}\wedge U_{i+k}\wedge \sqrt{n}\log\log n,\\
&\qquad\qquad\qquad U_{i+1}\vee\dots\vee U_{i+k-1}\le y\le  U_{i}\wedge U_{i+k}\wedge \sqrt{n}\log\log n)\ dx\ dy\\
&\qquad=\int_0^{\sqrt{n}\log\log n}\int_0^{\sqrt{n}\log\log n} \mathbb{P}(U_{i+1}\vee\dots\vee U_{i+k-1}\le x\le  U_{i}\wedge U_{i+k},\\
&\qquad\qquad\qquad\qquad\qquad\qquad U_{i+1}\vee\dots\vee U_{i+k-1}\le y\le  U_{i}\wedge U_{i+k})\ dx\ dy\\
&\qquad=2\int_0^{\sqrt{n}\log\log n}\int_0^{y} \mathbb{P}(U_{i+1}\vee\dots\vee U_{i+k-1}\le x) \mathbb{P}(U_{i}\wedge U_{i+k}\ge y)\ dx\ dy\\
&\qquad=2\int_0^{\sqrt{n}\log\log n}\int_0^{y} \left(\frac{x}{1+x}\right)^{k-1} \frac{1}{(1+y)^2}\ dx\ dy.\\
\end{align*}
Note that
$$2\int_0^{\sqrt{n}\log\log n}\int_0^{y}\frac{1}{(1+y)^2}\ dx\ dy=2\int_0^{\sqrt{n}\log\log n}\frac{y}{(1+y)^2}\ dy=\log n+2\log\log\log n+O(1),\\$$
and
\begin{align*}
&2\int_0^{\sqrt{n}\log\log n}\int_0^{y}\left(1-\frac{k-1}{x+1}\right)\left(\frac{1}{1+y}\right)^2\ dx\ dy\\
&\qquad	=2\int_0^{\sqrt{n}\log\log n}\frac{y-(k-1)\log (y+1)}{(1+y)^2}\ dy=\log n+2\log\log\log n+O(1).\\
\end{align*}
It follows from \eqref{Central limit theorem for an approximated version: analysis inequality} that
$$\E\left[\left((U_{i}\wedge U_{i+k}\wedge \sqrt{n}\log\log n-U_{i+1}\vee\dots\vee U_{i+k-1})^+\right)^2\right]=\log n+2\log\log\log n+O(1)$$
and
\begin{align*}
\Var(n\widetilde{L}^k_{i}) &= \Var\left((U_{i}\wedge U_{i+k}\wedge \sqrt{n}\log n-U_{i+1}\vee\dots\vee U_{i+k-1})^+\right)\\
&=\log n+2\log\log\log n+O(k)-\left(\frac{1}{k}+O\left(\frac{1}{\sqrt{n}\log\log n}\right)\right)^2 \\
&=\log n+2\log\log n+O(1).
\end{align*}
    
For the covariance, we have
\begin{equation*}
\begin{split}
0&\le\E[(U_{i}\wedge U_{i+k}\wedge \sqrt{n}\log\log n-U_{i+1}\vee\dots\vee U_{i+k-1})^+\\
&\qquad \times (U_{i'}\wedge U_{i'+k'}\wedge \sqrt{n}\log\log n-U_{i'+1}\vee\dots\vee U_{i'+k'-1})^+]\\
&\le\E\left[(U_{i}\wedge U_{i+k}\wedge \sqrt{n}\log\log n)(U_{i'}\wedge U_{i'+k'}\wedge \sqrt{n}\log\log n)\right].
\end{split}
\end{equation*}
If $(i,k)\neq(i',k')$, then $i$, $i+k$, $i'$, $i'+k'$ include three or four distinct values. If $i$, $i+k$, $i'$, $i'+k'$ include three distinct values, we have
\begin{equation*}
\begin{split}
&\E\left[(U_{i}\wedge U_{i+k}\wedge \sqrt{n}\log\log n)(U_{i'}\wedge U_{i'+k'}\wedge\sqrt{n}\log\log n)\right]\\
&\qquad=\E\left[(U_{1}\wedge U_{2}\wedge \sqrt{n}\log\log n)(U_{2}\wedge U_{3}\wedge \sqrt{n}\log\log n)\right]\\
&\qquad=\int_0^{\infty}\int_0^{\infty}\mathbb{P}(U_{1}\wedge U_{2}\wedge \sqrt{n}\log\log n\ge x, U_{2}\wedge U_{3}\wedge \sqrt{n}\log\log n\ge y)\ dx\ dy\\
&\qquad=\int_0^{\sqrt{n}\log\log n}\int_0^{\sqrt{n}\log\log n}\mathbb{P}(U_{1}\wedge U_{2}\ge x, U_{2}\wedge U_{3}\ge y)\ dx\ dy\\
&\qquad=2\int_0^{\sqrt{n}\log\log n}\int_0^{y}\mathbb{P}(U_{1}\ge x, U_{2}\ge y,U_{3}\ge y)\ dx\ dy\\
&\qquad=2\int_0^{\sqrt{n}\log\log n}\int_0^{y}\frac{1}{1+x}\frac{1}{(1+y)^2}\ dx\ dy\\
&\qquad=O(1).			
\end{split}
\end{equation*}
If $i$, $i+k$, $i'$, $i'+k'$ include four distinct values, we have
\begin{equation*}
\begin{split}
&\E\left[(U_{i}\wedge U_{i+k}\wedge \sqrt{n}\log\log n)(U_{i'}\wedge U_{i'+k'}\wedge\sqrt{n}\log\log n)\right]\\
&\qquad=\E\left[(U_{1}\wedge U_{2}\wedge \sqrt{n}\log\log n)(U_{3}\wedge U_{4}\wedge \sqrt{n}\log\log n)\right]\\
&\qquad=\left(\int_0^{\infty}\mathbb{P}(U_{1}\wedge U_{2}\wedge \sqrt{n}\log\log n\ge x)\ dx\right)^2\\
&\qquad=\left(\int_0^{\sqrt{n}\log\log n}\mathbb{P}(U_{1}\ge x, U_{2}\ge x)\ dx\right)^2\\
&\qquad=\left(\int_0^{\sqrt{n}\log\log n}\frac{1}{(1+x)^2}\ dx\right)^2\\
&\qquad=O(1).			
\end{split}
\end{equation*}
It follows that
$$\Cov\left(n\widetilde{L}^k_{i},n\widetilde{L}^{k'}_{i'}\right)=O(1)-\left(\frac{1}{k}+O\left(\frac{1}{\sqrt{n}\log\log n}\right)\right)\left(\frac{1}{k'}+O\left(\frac{1}{\sqrt{n}\log\log n}\right)\right)=O(1),$$
which completes the proof.
\end{proof}

\subsection{Central limit theorem for the approximate branch lengths}\label{Subsection: central limit theorem for an approximated version}

We prove a central limit theorem with the random variables $\widetilde{L}^k_{i}$.
\begin{proposition}\label{Proposition: central limit theorem for an approximated version} 
Let $\widetilde{L}^k_{i}$ be defined as in \eqref{Central limit theorem for an approximation: definition of widetilde L^k_i}. Then
\begin{equation}\label{LtildeCLT}
\sqrt{\frac{n}{\log n}}\left\{ \sum_{i=1}^{n-k-1} \widetilde{L}^k_{i}-\frac{1}{k}\right\}_{k=1}^K \Rightarrow N\left(0,I_{K\times K}\right).
\end{equation}
\end{proposition}

\begin{proof}
We consider the sequence
$$
\left\{\sum_{i=1}^{n-k-1} n\widetilde{L}^k_{i}\right\}_{k=1}^{K}
$$
and the linear combination
\begin{equation}\label{Central limit theorem for an approximated version: linear combination}
\sum_{k=1}^{K}\sum_{i=1}^{n-k-1} x_k n\widetilde{L}^k_{i}, 
\end{equation}
where $\{x_k\}_{k=1}^K$ is a sequence of nonzero real numbers. Recall from \eqref{Central limit theorem for an approximation: definition of widetilde L^k_i} that
$$
n\widetilde{L}^k_{i}=(U_{i}\wedge U_{i+k}\wedge \sqrt{n}\log\log n-U_{i+1}\vee\dots\vee U_{i+k-1})^{+}
$$ 
and the random variables $U_i$ are independent. Therefore, $x_kn\widetilde{L}^k_{i}$ and $x_{k'}n\widetilde{L}^{k}_{i}$ are dependent only if $|i-i'|\le K$. If we order the pairs $(i,k)$ lexicologically, so that $(i,k)\le(i',k')$ if $i<i'$, or if $i=i'$ and $k\le k'$, then $x_kn\widetilde{L}^k_{i}$ and $x_{k'}n\widetilde{L}^k_{i}$ are independent if $(i,k)$ and $(i',k')$ are at least $K(K+1)$ apart in this total ordering. We apply the $m$-dependent central limit theorem for a triangular array (see Theorem 1.1 of \cite{janson2021central}) to 
$$
\left\{x_kn\widetilde{L}^k_{i}-\E\left[x_kn\widetilde{L}^k_{i}\right]\right\}_{1\le k\le K, 1\le i\le n-k-1}.
$$
Let
$$\sigma_n^2=\Var\left(	\sum_{k=1}^{K}\sum_{i=1}^{n-k-1} x_k n \widetilde{L}^k_{i}\right).$$
It suffices to check the following condition:
\begin{equation}\label{Central limit theorem for an approximated version: clt condition}
\lim_{n\rightarrow\infty}\frac{1}{\sigma_n^2}\sum_{k=1}^{K}\sum_{i=1}^{n-k-1}\E\left[\left|x_kn\widetilde{L}^k_{i}-\E\left[x_kn\widetilde{L}^k_{i}\right]\right|^2\mathbbm{1}_{\left\{\left|x_kn\widetilde{L}^k_{i}-\E\left[x_kn\widetilde{L}^k_{i}\right]\right|> \epsilon\sigma_n\right\}}\right]=0,\text{ for all }\epsilon>0.
\end{equation}
By Lemma \ref{Lemma: moment computation}, we have
$$\E\left[nx_k\widetilde{L}^{k}_{i}\right]=\frac{x_k}{k}+O\left(\frac{1}{\sqrt{n}\log\log n}\right),$$
and
\begin{equation*}
\begin{split}
\sigma_n^2&=\sum_{k,k'=1}^{K}\sum_{i=1}^{n-k-1}\sum_{i'=1,|i-i'|\le K}^{n-k'-1}\Cov\left(x_kn \widetilde{L}^k_{i},x_{k'}n \widetilde{L}^{k'}_{i'}\right)\\
&=\sum_{k=1}^K\sum_{i=1}^{n-k-1} x_k^2 \Var(n \widetilde{L}^k_{i})+\sum_{k,k'=1}^{K}\sum_{i=1}^{n-k-1}\sum_{i'=1,|i-i'|\le K,(i,k)\neq (i',k')}^{n-k'-1}x_kx_{k'}\Cov\left(n \widetilde{L}^k_{i},n \widetilde{L}^{k'}_{i'}\right)\\
&=\left(\sum_{k=1}^{K}x_k^2\right)n\log n+O(n\log\log\log n).
\end{split}
\end{equation*}
Since the random variable $n\widetilde{L}^k_{1},\dots, n\widetilde{L}^k_{n-k-1}$ are identical in distribution for a fixed $k$, and their expectation is bounded by some constant independent of $n$, to check \eqref{Central limit theorem for an approximated version: clt condition}, it suffices to check that
\begin{equation}\label{Central limit theorem for an approximated version: clt condition simplified}
\lim_{n\rightarrow\infty}\frac{1}{\log n}\E\left[\left(n\widetilde{L}^k_{i}\right)^2\mathbbm{1}_{\left\{n\widetilde{L}^k_{i}> \epsilon\sqrt{n\log n}\right\}}\right]=0,\qquad\text{ for all }\epsilon>0, 1\le k\le K.
\end{equation}
Notice that for any $\epsilon>0$, for $n$ sufficiently large, we have
$$n\widetilde{L}^k_{i}=(U_{i}\wedge U_{i+k}\wedge \sqrt{n}\log\log n-U_{i+1}\vee\dots\vee U_{i+k-1})^{+}\le \sqrt{n}\log\log n\le\epsilon\sqrt{n\log n}.$$
Therefore, the condition \eqref{Central limit theorem for an approximated version: clt condition simplified} is trivially satisfied and we have
\begin{equation*}
\frac{1}{\sqrt{n\log n}}\left(\sum_{k=1}^{K}\sum_{i=1}^{n-k-1} x_kn \widetilde{L}^k_{i}-\E\left[\sum_{k=1}^{K}\sum_{i=1}^{n-k-1} x_kn  \widetilde{L}^k_{i}\right]\right) \Rightarrow N\left(0,\sum_{k=1}^K x_k^2\right).
\end{equation*}
Since $\{x_k\}_{k=1}^K$ is an arbitrary sequence of nonzero real numbers, by the Cramér–Wold theorem, we have
\begin{equation}\label{Central limit theorem for an approximated version: clt}
\left\{ \frac{1}{\sqrt{n\log n}}\sum_{i=1}^{n-k-1}\left(n \widetilde{L}^k_{i}-\E\left[n\widetilde{L}^k_{i}\right]\right)\right\}_{k=1}^K \Rightarrow N\left(0,I_{K\times K}\right).
\end{equation}
By Lemma \ref{Lemma: moment computation}, we have
$$\E\left[n\widetilde{L}^k_{i}\right]=\frac{1}{k}+O\left(\frac{1}{\sqrt{n}\log\log n}\right).$$
Plugging this into \eqref{Central limit theorem for an approximated version: clt}, we obtain the result in \eqref{LtildeCLT}.
\end{proof}

\section{Error bounds}\label{Section: error bounds}
Recall the different approximations we have for $L^k_{i,n,T}/N_T$ in \eqref{Central limit theorem for an approximation: definition of widetilde L^k_i}, \eqref{Central limit theorem for an approximation: definition of widetilde L^k_i,n}, and \eqref{Central limit theorem for an approximation: definition of widetilde L^k_i,n,T}. In this section, we bound the errors in these approximations. In Section \ref{Subsection: error between widetilde L^k_i and widetilde L^k_i,n}, we bound the error from taking $n\rightarrow\infty$ (after having taken $T\rightarrow\infty$) by comparing $\widetilde{L}^k_{i}$ and $\widetilde{L}^k_{i,n}$. In Section \ref{Subsection: error between widetilde L^k_i,n and widetilde L^k_i,n,T}, we bound the error from taking $T\rightarrow\infty$ by comparing $\widetilde{L}^k_{i,n}$ and $\widetilde{L}^k_{i,n,T}$. In Section \ref{Subsection: Error between widetilde L^k_i,n,T and L^k_i,n,T/N_T}, we bound the error resulting from truncation and from approximating $T/(Q_{n,T} N_T)$ with $1/n$ by comparing $\widetilde{L}^k_{i,n,T}$ and $L^{k}_{i,n,T}/N_T$. Finally, we use these bounds to prove Theorem \ref{Theorem} in Section \ref{Subsection: Proof of Theorem}.

To bound these errors, we use a similar approach as in Section 1.1 of the supplemental information to \cite{johnson2023estimating} by considering a maximal coupling of random variables (see section 4.4 of chapter 1 of \cite{thorisson}). If $X$ and $Y$ are random variables with probability density functions $f$ and $g$ respectively, a maximal coupling of $X$ and $Y$ is a coupling such that $P(X = Y) = \int_{-\infty}^{\infty} f(x) \wedge g(x) \ dx$ and $P(X \neq Y) = \frac{1}{2} \int_{-\infty}^{\infty} |f(x) - g(x)| \ dx.$ 

\subsection{Error between $\widetilde{L}^k_{i}$ and  $\widetilde{L}^k_{i,n}$}\label{Subsection: error between widetilde L^k_i and widetilde L^k_i,n}

Consider the coupling of $Q_{n}/n$ and $1/W$ such that
$$\mathbb{P}(Q_{n}/n=1/W)=\int_{0}^\infty f_{Q_{n}/n}(q)\wedge f_{1/W}(q) \ dq.$$
On the event $Q_{n}/n=1/W=q$, we can couple $U_{i,n}$ and $U_i$ such that
$$\mathbb{P}(U_{i,n}=U_{i}|Q_{n}/n=1/W=q)=\int_0^\infty f_{U_i}(u)\wedge f_{U_{i,n}|Q_{n}/n=q}(u)\ du.$$ 
On the event ${Q}_{n}/n\neq1/W$, take $U_{i,n}$ and $U_{i}$ to be arbitrary random variables with the prescribed conditional densities.
We now bound the error from the approximation under this coupling.

\begin{lemma}\label{Lemma: coupling of Q_n/n and 1/W}
Under the coupling described above, we have
$$
\mathbb{P}(Q_{n}/n \neq 1/W) = O\left(\frac{1}{n}\right),
$$
and 
$$
\mathbb{P}(U_{i,n}\neq U_i)=O\left(\frac{1}{n}\right).
$$
\end{lemma}

\begin{proof}
The first claim is proved in Lemma 5 in the supplemental information of \cite{johnson2023estimating}. Note that $\widetilde{Q}_{n,\infty}$ in \cite{johnson2023estimating} corresponds to $Q_{n}/n$ here. Although Lemma 5 of \cite{johnson2023estimating} is stated for the supercritical case, the formulas for $f_{Q_{n}/n}$ and $f_{1/W}$ are the same in the critical case.

For the second claim, by comparing $f_{U_{i,n}|Q_{n}/n=q}$ in \eqref{Approximation of H_i,n,T for large n and T: f_U_i,n|Q_n/n=q} and $f_{U_{i}}$ in \eqref{Approximation of H_i,n,T for large n and T: f_U_i}, we have
\begin{equation}\label{Error between widetilde L^k_i and  widetilde L^k_i,n: comparing density}
|f_{U_{i,n}|Q_{n}/n=q}(u)-f_{U_{i}}(u)|= \frac{1}{nq}\frac{1}{(1+u)^2}\mathbbm{1}_{\{0\le u\le nq\}}+\frac{1}{(1+u)^2}\mathbbm{1}_{\{u>nq\}}.  
\end{equation}
It follows that
$$
\mathbb{P}(U_{i,n}\neq U_i|Q_{n}/n=1/W=q)\le \frac{1}{2}\int_0^{nq} \frac{1}{nq}\frac{1}{(1+u)^2}\ du+\frac{1}{2}\int_{nq}^\infty \frac{1}{(1+u)^2}\ du\le \frac{C}{nq}.
$$
Integrating this against $f_{Q_n/n}\wedge f_{1/W}$, we have
$$
\mathbb{P}(U_{i,n}\neq U_i,Q_{n}/n=1/W)\le C\int_0^\infty \frac{f_{Q_n/n}(q)\wedge f_{1/W}(q)}{nq}\ dq\le
C\int_0^\infty \frac{e^{-1/q}}{nq^3}\ dq\le \frac{C}{n}.
$$
The second claim follows from this bound and the first claim.
\end{proof}

\begin{lemma}\label{Lemma: widetilde L^k_i,n-widetilde L^k_i}
Let $\widetilde{L}^k_{i}$ and $\widetilde{L}^k_{i,n}$ be defined in \eqref{Central limit theorem for an approximation: definition of widetilde L^k_i} and \eqref{Central limit theorem for an approximation: definition of widetilde L^k_i,n}. Under the coupling described above, we have
$$\sum_{i=1}^{n-k-1} \left(\widetilde{L}^k_{i}-\widetilde{L}^k_{i,n}\right)=O_p\left(\frac{\log\log n}{\sqrt{n}}\right).$$
\end{lemma}

\begin{proof}
By the triangle inequality and Markov's inequality, it suffices to show that
\begin{equation}\label{Error between widetilde L^k_i and widetilde L^k_i,n}
\begin{split}
\mathbb{P}(Q_{n}/n\neq 1/W)&=o(1),\\
\mathbb{P}\left(Q_{n}/n=1/W\le \frac{\log\log n}{\sqrt{n}}\right) &= o(1),\\
 \E\left[\left|\widetilde{L}^k_{i,n}-\widetilde{L}^k_{i}\right|\mathbbm{1}_{\left\{Q_{n}/n=1/W > (\log \log n)/\sqrt{n}\right\}}\right]&=O\left(\frac{\log\log n}{n^{3/2}}\right).
\end{split} 
\end{equation}

The first claim in \eqref{Error between widetilde L^k_i and widetilde L^k_i,n} follows from Lemma \ref{Lemma: coupling of Q_n/n and 1/W}.
The second claim in \eqref{Error between widetilde L^k_i and widetilde L^k_i,n} follows from the observation that
$$
\mathbb{P}\left(Q_{n}/n=1/W\le \frac{\log\log n}{\sqrt{n}}\right)\le \mathbb{P}\left(1/W\le \frac{\log\log n}{\sqrt{n}}\right) = o(1).
$$

For the third claim in \eqref{Error between widetilde L^k_i and widetilde L^k_i,n}, notice that
\begin{equation}\label{Labs}
\left|\widetilde{L}^k_{i,n}-\widetilde{L}^k_{i}\right|\le\frac{1}{n}\sum_{j=i}^{i+k} \big|U_{j,n}\wedge\sqrt{n}\log\log n-U_{j}\wedge\sqrt{n}\log\log n\big|.
\end{equation}
We have
\begin{align} \label{Error between widetilde L^k_i and widetilde L^k_i,n, last claim}
&\big|U_{i,n}\wedge\sqrt{n}\log\log n-U_{i}\wedge \sqrt{n}\log\log n\big| \nonumber \\
&\qquad=\big|U_{i,n}\mathbbm{1}_{\{U_{i,n}\le \sqrt{n}\log\log n\}}+\sqrt{n}(\log\log n)\mathbbm{1}_{\{U_{i,n}> \sqrt{n}\log\log n\}} \nonumber \\
&\qquad\qquad-U_{i}\mathbbm{1}_{\{U_{i}\le \sqrt{n}\log\log n\}}-\sqrt{n}(\log\log n)\mathbbm{1}_{\{U_{i}> \sqrt{n}\log\log n\}}\big|\nonumber \\
&\qquad\le\sqrt{n}(\log\log n)\big|\mathbbm{1}_{\{U_{i,n}> \sqrt{n}\log\log n\}}-\mathbbm{1}_{\{U_{i}> \sqrt{n}\log\log n\}}\big| \nonumber \\
&\qquad\qquad+\big|U_{i,n}\mathbbm{1}_{\{U_{i,n}\le \sqrt{n}\log\log n\}}-U_{i}\mathbbm{1}_{\{U_{i}\le \sqrt{n}\log\log n\}}\big|.
\end{align}
To bound the first term on the right-hand side of \eqref{Error between widetilde L^k_i and widetilde L^k_i,n, last claim}, we use \eqref{Error between widetilde L^k_i and  widetilde L^k_i,n: comparing density} in the third line and the fact that $nq> \sqrt{n}\log\log n$ for $q>(\log\log n)/\sqrt{n}$ in the fourth line to get
\begin{align}\label{UiUin1}
&\sqrt{n}(\log\log n)\mathbb{E}\left[\left|\mathbbm{1}_{\{U_{i,n}>\sqrt{n}\log\log n\}}-\mathbbm{1}_{\{U_{i}>\sqrt{n}\log\log n\}}\right|\mathbbm{1}_{\left\{Q_{n}/n=1/W>(\log\log n)/\sqrt{n}\right\}}\right]\nonumber\\
&\qquad\le\sqrt{n}(\log\log n)\int_{(\log\log n)/\sqrt{n}}^{\infty}\int_{\sqrt{n}\log\log n}^{\infty} \left|f_{U_{i,n}|Q_{n}/n=q}(u)-f_{U_i}(u)\right|\left|f_{1/W}(q)\wedge f_{Q_{n}/n}(q)\right|\ du\ dq \nonumber \\
&\qquad\le\sqrt{n}(\log\log n)\int_{(\log\log n)/\sqrt{n}}^\infty\left(\int_{\sqrt{n}\log\log n}^{nq}\frac{1}{nq}\frac{1}{(1+u)^2}\ du+\int_{nq}^\infty\frac{1}{(1+u)^2}\ du\right)\frac{e^{-1/q}}{q^2}\ dq \nonumber \\
&\qquad\le\sqrt{n}(\log\log n)\int_{(\log\log n)/\sqrt{n}}^\infty\left(\frac{1}{nq}\frac{1}{\sqrt{n}\log\log n}+\frac{1}{nq}\right)\frac{e^{-1/q}}{q^2}\ dq \nonumber \\
&\qquad\le\sqrt{n}(\log\log n)\int_{(\log\log n)/\sqrt{n}}^\infty\frac{2}{nq}\frac{e^{-1/q}}{q^2}\ dq \nonumber \\
&\qquad=O\left(\frac{\log\log n}{n^{1/2}}\right).
\end{align}
To bound the second term on the right-hand side of \eqref{Error between widetilde L^k_i and widetilde L^k_i,n, last claim}, we use \eqref{Error between widetilde L^k_i and  widetilde L^k_i,n: comparing density} in the third line to get
\begin{align}\label{Error bound: U_{i,n,infty} and  U_{i}}
&\E\left[\left|U_{i,n}\mathbbm{1}_{\{U_{i,n}\le \sqrt{n}\log\log n\}}-U_{i}\mathbbm{1}_{\{U_{i}\le \sqrt{n}\log\log n\}}\right|\mathbbm{1}_{\left\{Q_{n}/n= 1/W>(\log\log n)/\sqrt{n}\right\}}\right]\nonumber\\
&\qquad\le\int_{(\log\log n)/\sqrt{n}}^{\infty}\int_0^{\sqrt{n}\log\log n} u\left|f_{U_{i,n}|Q_{n}/n=q}(u)-f_{U_i}(u)\right|\left|f_{1/W}(q)\wedge f_{Q_{n}/n}(q)\right|\ du\ dq\nonumber \\
&\qquad= \int_{(\log\log n)/{\sqrt{n}}}^\infty \int_0^{\sqrt{n}\log\log n} u \cdot \frac{1}{nq}\frac{1}{(1+u)^2}\frac{e^{-1/q}}{q^2} \ du\ dq \nonumber \\
&\qquad\le\frac{1}{n}\left(\int_0^\infty \frac{e^{-1/q}}{q^3}\ dq\right)\left(\int_0^{\sqrt{n}\log\log n}\frac{u}{(1+u)^2} \ du\right) \nonumber \\
&\qquad=O\left(\frac{\log n}{n}\right).
\end{align}
It follows from \eqref{Labs}, \eqref{Error between widetilde L^k_i and widetilde L^k_i,n, last claim}, \eqref{UiUin1}, and \eqref{Error bound: U_{i,n,infty} and U_{i}} that 
$$
\E\left[\left|\widetilde{L}^k_{i,n}-\widetilde{L}^k_{i}\right|\mathbbm{1}_{\{Q_{n}/n= 1/W>(\log\log n)/\sqrt{n}\}}\right]=O\left(\frac{\log\log n}{n^{3/2}}\right)
$$
and the last claim is proved.
\end{proof}

\subsection{Error between $\widetilde{L}^k_{i,n}$ and $\widetilde{L}^k_{i,n,T}$}\label{Subsection: error between widetilde L^k_i,n and widetilde L^k_i,n,T}

Consider the coupling of ${Q}_{n,T}$ and $Q_{n}$ such that
$$
\mathbb{P}({Q}_{n,T}=Q_{n})=\int_{0}^\infty f_{{Q}_{n,T}}(q)\wedge f_{Q_{n}}(q) \ dq.
$$
On the event ${Q}_{n,T}={Q}_{n}=q$, we can couple $U_{i,n,T}$ and $U_{i,n}$ such that
$$
\mathbb{P}(U_{i,n,T}=U_{i,n}|{Q}_{n,T}={Q}_{n}=q)=\int_0^\infty f_{U_{i,n,T}|Q_{n,T}=q}(u)\wedge f_{U_{i,n}|Q_{n}=q}(u)\ du.
$$ 
On the event  ${Q}_{n,T}\neq{Q}_{n}$, take $U_{i,n,T}$ and $U_{i,n}$ to be arbitrary random variables with the prescribed conditional densities. We now bound the error from the approximation under this coupling.  Because we are interested in applying these results when the condition \eqref{Introduction: sequence condition} holds, we may restrict our attention to the case when $n < T$.

\begin{lemma}\label{Lemma: coupling of Q_n,T and Q_n}
Under the coupling described above, and the assumption that $n<T$, we have
$$
\mathbb{P}(Q_{n,T} \neq Q_n) = O\left(\frac{n}{T}\right),
$$
and
$$
\mathbb{P}(U_{i,n,T} \neq U_{i,n}) = O\left(\frac{n}{T}\right).
$$
\end{lemma}
\begin{proof}
By comparing $f_{Q_{n,T}}$ in \eqref{Approximation of H_i,n,T for large T: f_U_i,n,T|Q_n,T=q} and $f_{Q_{n}}$ in \eqref{Approximation of H_i,n,T for large T: f_Q_n}, we have
$$
f_{Q_{n,T}}(q)=f_{Q_{n}}(q)\Big(1+O\Big(\frac{n}{T}\Big)\Big),\qquad\text{ for } q\in(0,1/\delta_T).
$$
Therefore, 
\begin{equation*}
\begin{split}
\mathbb{P}(Q_{n,T}\neq Q_{n})&=\int_0^\infty \left|f_{Q_{n,T}}(q)-f_{Q_{n}}(q)\right| \ dq\\
&\le C\int_0^{1/\delta_T} \frac{n}{T}f_{{Q}_{n}}(q)\ dq+\int_{1/\delta_T} ^\infty f_{Q_n}(q)\ dq\\
&\le C\int_0^{\infty} \frac{n}{T}f_{{Q}_{n}}(q)\ dq+\int_{1/\delta_T} ^\infty f_{Q_n}(q)\ dq\\
&=\frac{Cn}{T}+1-\left(\frac{T+1}{T+2}\right)^n\\
&=O\left(\frac{n}{T}\right),
\end{split}
\end{equation*}
which proves the claim.

By comparing and $f_{U_{i,n,T}|Q_{n,T}=q}$ in \eqref{Approximation of H_i,n,T for large T: f_U_i,n,T|Q_n,T=q} and $f_{U_{i,n}|Q_{n}=q}$ in \eqref{Approximation of H_i,n,T for large T: f_U_i,n|Q_n=q} we have 
\begin{equation}\label{Error between widetilde L^k_i,n and widetilde L^k_i,n,T: comparing density}
f_{U_{i,n,T}|Q_{n,T}=q}(u)=f_{U_{i,n}|Q_{n}=q}(u)\Big(1+O\Big(\frac{1}{T}\Big)\Big).
\end{equation}
Since the bound of $1 + O(1/T)$ does not depend on $u$, it follows that 
$$
\mathbb{P}(U_{i,n,T}\neq U_{i,n}|Q_{n,T}=Q_n=q) = O\left(\frac{1}{T}\right).
$$
Since the right hand side does not depend on $q$, it follows that
$$
\mathbb{P}(U_{i,n,T}\neq U_{i,n},Q_{n,T}=Q_n)=O\left(\frac{1}{T}\right).
$$
The second claim follows from this result and the first claim.
\end{proof}

\begin{lemma}\label{Lemma: widetilde L^k_i,n,T-widetilde L^k_i,n}
Under the coupling described above, we have
\begin{equation}\label{Error between widetilde L^k_i,n and widetilde L^k_i,n,T}
\E\left[\left|\widetilde{L}^k_{i,n,T}-\widetilde{L}^k_{i,n}\right|\mathbbm{1}_{\{Q_{n,T}= Q_{n}\}}\right]=O\left(\frac{\log n}{nT}\right).
\end{equation} 
\end{lemma}

\begin{proof}
Note that
\begin{equation}\label{LLT1}
\left|\widetilde{L}^k_{i,n,T}-\widetilde{L}^k_{i,n}\right|\le\frac{1}{n}\sum_{j=i}^{i+k} \big|U_{j,n,T}\wedge\sqrt{n}\log\log n-U_{j,n}\wedge\sqrt{n}\log\log n\big|,
\end{equation}
Notice that, as in \eqref{Error between widetilde L^k_i and widetilde L^k_i,n, last claim}
\begin{align}\label{Error between widetilde L^k_i,n and widetilde L^k_i,n,T, last claim}
&\big|U_{i,n,T}\wedge\sqrt{n}\log\log n-U_{i,n}\wedge \sqrt{n}\log\log n\big|\nonumber \\
&\qquad\le\sqrt{n}(\log\log n)\big|\mathbbm{1}_{\{U_{i,n,T}> \sqrt{n}\log\log n\}}-\mathbbm{1}_{\{U_{i,n}> \sqrt{n}\log\log n\}}\big| \nonumber \\
&\qquad\qquad+\big|U_{i,n,T}\mathbbm{1}_{\{U_{i,n,T}\le \sqrt{n}\log\log n\}}-U_{i,n}\mathbbm{1}_{\{U_{i,n}\le \sqrt{n}\log\log n\}}\big|.
\end{align}
To bound the first term on the right-hand side \eqref{Error between widetilde L^k_i,n and widetilde L^k_i,n,T, last claim}, notice that $U_{i,n,T}\le Q_{n,T}$ by construction. Therefore, $U_{i,n,T}>\sqrt{n}\log\log n$ occurs only if $Q_{n,T}>\sqrt{n}\log\log n$. Similar arguments hold for $U_{i,n}$ and $Q_{n}$. Using \eqref{Error between widetilde L^k_i,n and widetilde L^k_i,n,T: comparing density} in the fourth line and assuming $n\ge2$ in the last line, we have
\begin{align}\label{LLT3}
&\sqrt{n}(\log\log n)\mathbb{E}\left[|\mathbbm{1}_{\{U_{i,n,T}>\sqrt{n}\log\log n\}}-\mathbbm{1}_{\{U_{i,n}>\sqrt{n}\log\log n\}}|\mathbbm{1}_{\left\{Q_{n,T}=Q_{n}\right\}}\right] \nonumber \\
&\qquad=\sqrt{n}(\log\log n)\mathbb{E}\left[|\mathbbm{1}_{\{U_{i,n,T}>(\log\log n)/\sqrt{n}\}}-\mathbbm{1}_{\{U_{i,n}>\sqrt{n}\log\log n\}}|\mathbbm{1}_{\left\{Q_{n,T}=Q_{n}>\sqrt{n}\log\log n\right\}}\right] \nonumber \\
&\qquad\le\sqrt{n}(\log\log n)\int_{\sqrt{n}\log\log n}^{\infty}\int_{\sqrt{n}\log\log n}^\infty \left|f_{Q_{n,T}}(q)\wedge f_{Q_{n}}(q)\right|\left|f_{U_{i,n,T}|Q_{n,T}=q}(u)-f_{U_{i,n}|Q_{n}=q}(u)\right|\ du\ dq \nonumber \\
&\qquad\le C\sqrt{n}(\log\log n)\int_{\sqrt{n}\log\log n}^\infty\int_{\sqrt{n}\log\log n}^q f_{Q_n}(q)\frac{1}{T}f_{U_{i,n}|Q_{n}=q}(u)\ du\ dq \nonumber \\
&\qquad= C\sqrt{n}(\log\log n)\int_{\sqrt{n}\log\log n}^\infty\int_{\sqrt{n}\log\log n}^q \frac{nq^{n-1}}{(1+q)^{n+1}}\frac{1}{T}\frac{1+q}{q}\frac{1}{(1+u)^2}\ du\ dq \nonumber \\
&\qquad\le C\frac{\sqrt{n}\log\log n}{T}\left(\int_{\sqrt{n}\log\log n}^\infty\frac{1}{(1+u)^2}\ du\right)\left(\int_{\sqrt{n}\log\log n}^\infty \frac{nq^{n-2}}{(1+q)^n}\ dq\right) \nonumber\\
&\qquad=O\left(\frac{1}{T}\right).
\end{align}
To bound the second term on the right-hand side of \eqref{Error between widetilde L^k_i,n and widetilde L^k_i,n,T, last claim}, we have
\begin{align}\label{LLT4}
&\E\left[\left|U_{i,n,T}\mathbbm{1}_{\{U_{i,n,T}\le \sqrt{n}\log\log n\}}-U_{i,n}\mathbbm{1}_{\{U_{i,n}\le \sqrt{n}\log \log n\}}\right|\mathbbm{1}_{\{Q_{n,T_n}= Q_{n}\}}\right] \nonumber \\
&\qquad\le\int_0^{\infty}\int_0^{\sqrt{n}\log\log n} u\left|f_{Q_{n,T}}(q)\wedge f_{Q_{n}}(q)\right|\left|f_{U_{i,n,T}|Q_{n,T}=q}(u)-f_{U_{i,n}|Q_{n}=q}(u)\right|\ du\ dq \nonumber \\
&\qquad\le C\int_0^{\infty}\int_0^{\sqrt{n}\log\log n} u f_{Q_{n}}(q)\frac{1}{T}f_{U_{i,n}|Q_{n}=q}(u)\ du\ dq \nonumber \\
&\qquad= C\int_0^\infty \int_0^{\sqrt{n}(\log\log  n)\wedge q} u \frac{nq^{n-1}}{(1+q)^{n+1}} \frac{1}{T}\frac{1+q}{q}\frac{1}{(1+u)^2} \ du\ dq \nonumber \\
&\qquad\le\frac{C}{T}\left(\int_0^\infty \frac{nq^{n-2}}{(1+q)^n}\ dq\right)\left(\int_0^{\sqrt{n}\log\log n}\frac{u}{(1+u)^2} \ du\right) \nonumber \\
&\qquad=O\left(\frac{\log n}{T}\right).
\end{align}
The result \eqref{Error between widetilde L^k_i,n and widetilde L^k_i,n,T} follows from \eqref{LLT1}, \eqref{Error between widetilde L^k_i,n and widetilde L^k_i,n,T, last claim}, \eqref{LLT3}, and \eqref{LLT4}.
\end{proof}

\subsection{Error between $\widetilde{L}^k_{i,n,T}$ and $L^{k}_{i,n,T}/N_T$}\label{Subsection: Error between widetilde L^k_i,n,T and L^k_i,n,T/N_T}
Recall the formula of $L^k_{i,n,T}/N_T$ in \eqref{Central limit theorem for an approximation: formula  L^k_i,n,T/N_T of L^k_i,n,T/N_T} and the definition of $\widetilde{L}^k_{i,n,T}$ in \eqref{Central limit theorem for an approximation: definition of widetilde L^k_i,n,T}.
There are two sources of error: the truncation at $\sqrt{n}\log\log n$ and the approximation of $T/\left(Q_{n,T}N_T\right)$ by $1/n$.  The following lemma gives some useful bounds.

\begin{lemma}\label{Lemma: widetilde{L}_{n,T}}
We have
\begin{align}\label{551}
\mathbb{P}\Big(\big(U_{i,n,T}\wedge U_{i+k,n,T}\wedge\sqrt{n}\log\log n&- U_{i+1,n,T}\vee\dots\vee U_{i+k-1,n,T}\big)^+ \nonumber \\
\neq \big(U_{i,n,T}\wedge U_{i+k,n,T}&-U_{i+1,n,T}\vee\dots\vee U_{i+k-1,n,T}\big)^+\Big)=O\left(\frac{1}{n(\log\log n)^2}\right)
\end{align}
and
\begin{align}\label{552}
\mathbb{P}\Big(\big(U_{k,n,T}\wedge{\sqrt{n}}\log\log n &-U_{1,n,T}\vee\dots\vee U_{k-1,n,T}\big)^+ \nonumber \\
\neq \big( U_{k,n,T}&-U_{1,n,T}\vee\dots\vee U_{k-1, n,T}\big)^+\Big)=O\left(\frac{1}{\sqrt{n}\log\log n}\right).
\end{align}
We also have
\begin{equation}\label{553}
\frac{T}{Q_{n,T}}-\frac{\E[N_T|Y_{n,T}]}{n}=O_p\left(\frac{T}{Q_{n,T}^2}+\frac{1}{Q_{n,T}}+\frac{1}{n}+\frac{T}{nQ_{n,T}}\right)
\end{equation}
and
\begin{equation}\label{554}
\frac{1}{n}\left(N_{T}-\E\left[N_{T}|Y_{n,T}\right]\right)= O_p\left(\frac{T}{n^{3/2}}\right).
\end{equation}
\end{lemma}

\begin{proof}
The event in \eqref{551} occurs only if $U_{i,n,T}\wedge U_{i+k,n,T}> \sqrt{n}\log\log n$. Using \eqref{Error between widetilde L^k_i,n and widetilde L^k_i,n,T: comparing density}, under the condition $n<T$, we have
\begin{equation*}
\begin{split}
\mathbb{P}(U_{i,n,T}\wedge U_{i+k,n,T}\ge \sqrt{n}\log\log n)&=\int_0^ {1/\delta_T} \left(\int_{\sqrt{n}\log\log n}^\infty f_{U_{i,n,T}|Q_{n,T}=q}(u)\ du\right)^2 f_{Q_{n,T}}(q)\ dq\\
&\le C\left(\int_0^ {1/\delta_T} \left(\int_{\sqrt{n}\log\log n}^\infty  f_{U_{i,n}|Q_{n}=q}(u)\ du\right)^2 f_{Q_{n}}(q)\ dq\right)\\
&\le C\left(\int_0^ {1/\delta_T} \frac{(1+q)^2}{q^2} \frac{1}{(1+\sqrt{n}\log\log n)^2} f_{Q_{n}}(q)\ dq\right)\\
&= \frac{C}{(1+\sqrt{n}\log\log n)^2}\left(\int_0^ {1/\delta_T} \frac{(1+q)^2}{q^2} \frac{nq^{n-1}}{(1+q)^{n+1}}\ dq\right)\\
&\le \frac{C}{(1+\sqrt{n}\log\log n)^2}\left(\int_0^\infty \frac{(1+q)^2}{q^2} \frac{nq^{n-1}}{(1+q)^{n+1}}\ dq\right)\\
&=\frac{Cn}{(n-2)(1+\sqrt{n}\log\log n)^2}\\
&\le \frac{C}{n(\log\log n)^2}.
\end{split}
\end{equation*}
Similarly, the probability in \eqref{552} is bounded by
\begin{equation*}
\begin{split}
\mathbb{P}(U_{k,n,T}\ge \sqrt{n}\log\log n)&=\int_0^ {1/\delta_T} \left(\int_{\sqrt{n}\log\log n}^\infty f_{U_{k,n,T}|Q_{n,T}=q}(u)\ du \right)f_{Q_{n,T}}(q)\ dq\\
&\le C\left(\int_0^ {1/\delta_T} \frac{1+q}{q} \frac{1}{1+\sqrt{n}\log\log n} f_{Q_{n}}(q)\ dq\right)\\
&\le \frac{C}{\sqrt{n}\log\log n}.
\end{split}
\end{equation*}

To prove \eqref{553}, recall from \eqref{Sampling from the coalescent point process: population size} that $N_T=\sum_{i=0}^n X_{i,n,T}-1$. By \eqref{Sampling from the coalescent point process: X_i,n,T p.m.f} and \eqref{Sampling from the coalescent point process: definition of delta_T}, we have for $0 \leq i \leq n$,
\begin{equation*}
\E\left[X_{i,n,T}|Y_{n,T}\right]
    =\frac{T+1}{TY_{n,T}+1}
\end{equation*}
and
\begin{equation*}
\Var(X_{i,n,T}|Y_{n,T})=\frac{T(T+1)(1-Y_{n,T})}{(TY_{n,T}+1)^2}.
\end{equation*}	
Recall that $Y_{n,T}=\delta_T Q_{n,T}$ by definition. It follows that
\begin{equation*}
\begin{split}
\frac{T}{Q_{n,T}}-\frac{\E[N_T|Y_{n,T}]}{n}&=
\frac{T}{Q_{n,T}}-\frac{1}{n}\left(\sum_{i=0}^{n}\E[X_{i,n,T}|Y_{n,T}]-1\right)\\
&=\frac{T}{Q_{n,T}}-\frac{1}{n}\left(\frac{(n+1)(T+1)}{TY_{n,T}+1}-1\right)\\
&=\frac{T}{Q_{n,T}}-\frac{(n+1)(T+1)}{n(TY_{n,T}+1)}+\frac{1}{n}\\
&=\frac{T}{Q_{n,T}}-\frac{(n+1)(T+1)}{n(\delta_T T Q_{n,T}+1)}+\frac{1}{n}\\
&=\frac{T}{Q_{n,T}}-\frac{(n+1)(T+1)^2}{ n(T Q_{n,T}+T+1)}+\frac{1}{n}\\
&=\frac{nT(T+1)-nQ_{n,T}(2T+1)-Q_{n,T}(T+1)^2}{nQ_{T,n}(TQ_{n,T}+T+1)}+\frac{1}{n}.
\end{split}
\end{equation*}
Taking the absolute value, we have
\begin{equation*}
\begin{split}
\left|\frac{T}{Q_{n,T}}-\frac{\E[N_T|Y_{n,T}]}{n}\right|&\le \frac{nT(T+1)+nQ_{n,T}(2T+1)+Q_{n,T}(T+1)^2}{nQ_{T,n}(TQ_{n,T}+T+1)}+\frac{1}{n}\\
&=O_p\left(\frac{T}{Q_{n,T}^2}+\frac{1}{Q_{n,T}}+\frac{1}{n}+\frac{T}{nQ_{n,T}}\right).
\end{split}
\end{equation*}

To prove \eqref{554}, note that under the assumption that $n \geq 3$, we have
\begin{equation*}
\begin{split}
\E\left[\Var\left(\frac{1}{n}\sum_{i=0}^{n}X_{i,n,T}\Big|Y_{n,T}\right)\right]
&=\frac{(n+1)}{n^2}\E\left[\Var\left(X_{i,n,T}\Big|Y_{n,T}\right)\right]\\
&=\frac{n+1}{n^2}\E\left[\frac{T(T+1)(1-Y_{n,T})}{(TY_{n,T}+1)^2}\right]\\
&\le \frac{n+1}{n^2}\E\left[\frac{T(T+1)}{T^2Y_{n,T}^2}\right]\\
&=\frac{(n+1)(T+1)}{n^2T} \int_0^1 \frac{1}{y^2} \frac{n\delta_Ty^{n-1}}{(y+\delta_T-y\delta_T)^{n+1}} \ dy\\
&=\frac{(n+1)(T+1)}{n^2T} \cdot \frac{ n(n-1)+2nT+2T^2}{(n-1)(n-2)}\\
&=O\left(\frac{T^2}{n^3}\right).\\
\end{split}
\end{equation*} 
Now \eqref{554} follows from this result and Chebyshev's inequality.
\end{proof}

\subsection{Proof of Theorem \ref{Theorem}}\label{Subsection: Proof of Theorem}
\begin{proof}
In view of \eqref{Recovering the branch length from the coalescent point process: definition of L^k_n,T} and Proposition \ref{Proposition: central limit theorem for an approximated version}, it suffices to show that for $k=1,2,\dots, K$, we have
$$\sum_{i=1}^{n-k-1}\widetilde{L}^k_{i}-\frac{1}{N_{T_n}}\sum_{i=0}^{n-k}L^k_{i,n,T_n}=o_p\left(\sqrt{\frac{\log n}{n}}\right).$$
Using the triangle inequality, it suffices to show
\begin{align}
&\sum_{i=1}^{n-k-1} (\widetilde{L}^k_{i}-\widetilde{L}^k_{i,n})=o_p\left(\sqrt{\frac{\log n}{n}}\right), \label{Proof of theorem: main1}\\
&\sum_{i=1}^{n-k-1} (\widetilde{L}^k_{i,n}-\widetilde{L}^k_{i,n,T_n})=o_p\left(\sqrt{\frac{\log n}{n}}\right), \label{Proof of theorem: main2} \\
&\sum_{i=1}^{n-k-1} \widetilde{L}^k_{i,n,T_n}-\frac{1}{N_{T_n}}\sum_{i=0}^{n-k}L^k_{i,n,T_n}=o_p\left(\sqrt{\frac{\log n}{n}}\right) \label{Proof of theorem: main3}.
\end{align} 

The claim \eqref{Proof of theorem: main1} is true by Lemma \ref{Lemma: widetilde L^k_i,n-widetilde L^k_i}. For \eqref{Proof of theorem: main2}, under the condition that $n/T_n\rightarrow0$ as $n\rightarrow\infty$, Lemma \ref{Lemma: coupling of Q_n,T and Q_n} implies that
$$
\mathbb{P}(Q_{n,T_n}\neq Q_{n})=o(1),
$$
and that
$$
\E\left[\left|\widetilde{L}^k_{i,n,T_n}-\widetilde{L}^k_{i,n}\right|\mathbbm{1}_{\{Q_{n,T_n}= Q_{n}\}}\right]=O\left(\frac{\log n}{n^2}\right)=o\left(\sqrt{\frac{\log n}{n^3}}\right).
$$
By Markov's inequality and the triangle inequality, we have
 $$
 \sum_{i=1}^{n-k} \left(\widetilde{L}^k_{i,n,T_n}-\widetilde{L}^k_{i,n}\right)=o_p\left(\sqrt{\frac{\log n}{n}}\right).
 $$
 
To prove \eqref{Proof of theorem: main3}, we have
\begin{align}\label{Proof of theorem: last claim}
&\sum_{i=1}^{n-k-1} \widetilde{L}^k_{i,n,T_n}-\frac{1}{N_{T_n}}\sum_{i=0}^{n-k}L^k_{i,n,T_n} \nonumber \\
&\qquad=\frac{1}{n}\sum_{i=1}^{n-k-1}\left(U_{i,n,T_n}\wedge U_{i+k,n,T_n}\wedge{\sqrt{n}\log\log n}-U_{i+1,n,T_n}\vee\dots\vee U_{i+k-1,n,T_n}\right)^{+} \nonumber \\ 
&\qquad\qquad-\frac{T_n}{Q_{n,T_n}N_{T_n}}\sum_{i=1}^{n-k-1}\left(U_{i,n,T_n}\wedge U_{i+k,n,T_n}-U_{i+1,n,T_n}\vee\dots\vee U_{i+k-1,n,T_n}\right)^{+} \nonumber \\
&\qquad\qquad-\frac{1}{N_{T_n}}\left(L^k_{0,n,T_n}+L^k_{n-k,n,T_n}\right) \nonumber \\
&\qquad=\Bigg(\frac{1}{n}\sum_{i=1}^{n-k-1}\left(U_{i,n,T_n}\wedge U_{i+k,n,T_n}\wedge{\sqrt{n}\log\log n}-U_{i+1,n,T_n}\vee\dots\vee U_{i+k-1,n,T_n}\right)^{+} \nonumber \\
&\qquad\qquad-\frac{T_n}{Q_{n,T_n}N_{T_n}}\sum_{i=1}^{n-k-1}\left(U_{i,n,T_n}\wedge U_{i+k,n,T_n}\wedge{\sqrt{n}\log\log n}-U_{i+1,n,T_n}\vee\dots\vee U_{i+k-1,n,T_n}\right)^{+}\Bigg) \nonumber \\
&\qquad\qquad+\Bigg(\frac{T_n}{Q_{n,T_n}N_{T_n}}\sum_{i=1}^{n-k-1}\left(U_{i,n,T_n}\wedge U_{i+k,n,T_n}\wedge{\sqrt{n}\log\log n}-U_{i+1,n,T_n}\vee\dots\vee U_{i+k-1,n,T_n}\right)^{+} \nonumber \\
&\qquad\qquad-\frac{T_n}{Q_{n,T_n}N_{T_n}}\sum_{i=1}^{n-k-1}\left(U_{i,n,T_n}\wedge U_{i+k,n,T_n}-U_{i+1,n,T_n}\vee\dots\vee U_{i+k-1,n,T_n}\right)^{+}\Bigg) \nonumber \\
&\qquad\qquad-\frac{1}{N_{T_n}}\left(L^k_{0,n,T_n}+L^k_{n-k,n,T_n}\right).
\end{align}
By Lemma \ref{Lemma: coupling of Q_n/n and 1/W} and Lemma \ref{Lemma: coupling of Q_n,T and Q_n}, we have
$$
\mathbb{P}\left(\frac{Q_{n,T_n}}{n}\neq \frac{1}{W}\right)\le\mathbb{P}\left(\frac{Q_{n,T_n}}{n}\neq \frac{Q_{n}}{n}\right)+\mathbb{P}\left(\frac{Q_{n}}{n}\neq \frac{1}{W}\right)=o(1).
$$
Therefore, we have
\begin{equation}\label{Proof of theorem: Q_n,T_n}
Q_{n,T_n}=O_p(n),\qquad\frac{1}{Q_{n,T_n}}=O_p\left(\frac{1}{n}\right),\qquad\frac{1}{Q^2_{n,T_n}}=O_p\left(\frac{1}{n^2}\right).
\end{equation}
The last two equations in Lemma \ref{Lemma: widetilde{L}_{n,T}} combined with \eqref{Proof of theorem: Q_n,T_n} imply that
\begin{equation}\label{Proof of theorem: T/Q-N_T/n}
\frac{T_n}{Q_{n,T_n}}-\frac{N_{T_n}}{n}= O_p\left(\frac{T_n}{n^{3/2}}\right).
 \end{equation}
By this equation and \eqref{Proof of theorem: Q_n,T_n}, we have
\begin{equation}\label{Proof of theorem: N_T}
\begin{split}
\frac{1}{N_{T_n}}&=\frac{1}{\frac{nT_n}{Q_{n,T_n}}+O_p\left(\frac{T_n}{\sqrt{n}}\right)}=\frac{1}{\frac{nT_n}{Q_{n,T_n}}}\cdot\frac{1}{1+O_p\left(\frac{Q_{n,T_n}}{n^{3/2}}\right)}\\
&=\frac{Q_{n,T_n}}{nT_n}\cdot\frac{1}{1+O_p\left(\frac{Q_{n,T_n}}{n^{3/2}}\right)}=O_p\left(\frac{1}{T_n}\right)\cdot\frac{1}{1+O_p\left(\frac{1}{\sqrt{n}}\right)}=O_p\left(\frac{1}{T_n}\right).
\end{split}
\end{equation}
Therefore, dividing \eqref{Proof of theorem: T/Q-N_T/n} by $N_{T_n}$, we have
\begin{equation}\label{TNQ}
\frac{T_n}{Q_{n,T_n}N_{T_n}}-\frac{1}{n}= O_p\left(\frac{1}{n^{3/2}}\right).
\end{equation}
Note that 
$$\sum_{i=1}^{n-k-1}\left(U_{i,n,T_n}\wedge U_{i+k,n,T_n}\wedge{\sqrt{n}\log\log n}-U_{i+1,n,T_n}\vee\dots\vee U_{i+k-1}\right)^{+}=n\sum_{i=1}^{n-k-1} \widetilde{L}^k_{i,n,T_n}=O_p(n)$$
by Proposition \ref{Proposition: central limit theorem for an approximated version} and equations \eqref{Proof of theorem: main1} and \eqref{Proof of theorem: main2}. Therefore, for the first term in \eqref{Proof of theorem: last claim}, we have
\begin{equation}\label{Proof of the main theorem: 1st term in the last claim}
\begin{split}
&\frac{1}{n}\sum_{i=1}^{n-k-1}\left(U_{i,n,T_n}\wedge U_{i+k,n,T_n}\wedge{\sqrt{n}\log\log n}-U_{i+1,n,T_n}\vee\dots\vee U_{i+k-1,n,T_n}\right)^{+}\\
&\quad-\frac{T_n}{Q_{n,T}N_{T_n}}\sum_{i=1}^{n-k-1}\left(U_{i,n,T_n}\wedge U_{i+k,n,T_n}\wedge{\sqrt{n}\log\log n}-U_{i+1,n,T_n}\vee \dots\vee U_{i+k-1,n,T_n}\right)^{+}=O_p\left(\frac{1}{\sqrt{n}}\right).
\end{split}
\end{equation}

For the second term in \eqref{Proof of theorem: last claim}, which accounts for the truncation at $\sqrt{n}\log\log n$, it follows from \eqref{551} that
\begin{align}\label{Proof of theorem: 2nd term in the last claim}
\mathbb{P}&\Bigg(\sum_{i=1}^{n-k-1}\left(U_{i,n,T_n}\wedge U_{i+k,n,T_n}\wedge{\sqrt{n}\log\log n}-U_{i+1,n,T_n}\vee\dots\vee U_{i+k-1, n, T_n}\right)^{+} \nonumber \\
&\quad\neq\sum_{i=1}^{n-k-1}\left(U_{i,n,T_n}\wedge U_{i+k,n,T_n}-U_{i+1,n,T_n}\vee\dots\vee U_{i+k-1, n, T_n}\right)^{+}\Bigg)=O\left(\frac{1}{(\log\log n)^2}\right)=o(1).
\end{align}

For the last term in \eqref{Proof of theorem: last claim}, which accounts for the lengths of the portions of the 0th and the $(n-k)$th branch supporting $k$ leaves, first note from the definition of $L^k_{0,n,T}$ in \eqref{Recovering the branch lenght from coalescent point process: definition of L^k_0,n,T} and the definition of $H_{k,n,T_n}$ that
$$\frac{L_{0,n,T_n}}{N_{T_n}} \leq \frac{H_{k,n,T_n}}{N_{T_n}} = \frac{T_n U_{k,n,T_n}}{Q_{n,T_n} N_{T_n}}.$$
By Lemma \ref{Lemma: coupling of Q_n/n and 1/W} and \ref{Lemma: coupling of Q_n,T and Q_n}, under the assumption that $n/T_n\rightarrow0$ as $n\rightarrow\infty$, we have
\begin{equation*}
\mathbb{P}\left(U_{i}\neq U_{i,n,T}\right)\le \mathbb{P}\left(U_{i}\neq U_{i,n}\right)+\mathbb{P}\left(U_{i,n}\neq U_{i,n,T}\right)=o(1),
\end{equation*}
which implies that $U_{k,n,T_n}$ is $O_p(1)$.  Combining this result with \eqref{TNQ} and applying the same argument to bound $L_{n-k,n,T_n}^k/N_{T_n}$, we get
\begin{equation}\label{Proof of theorem: 3rd term in the last claim}
\frac{1}{N_T}\left(L^k_{0,n,T_n}+L^k_{n-k,n,T_n}\right)=O_p\left(\frac{1}{n}\right).
\end{equation}
By \eqref{Proof of the main theorem: 1st term in the last claim}, \eqref{Proof of theorem: 2nd term in the last claim}, and \eqref{Proof of theorem: 3rd term in the last claim}, we know that the quantity in \eqref{Proof of theorem: last claim} is $O_p\left({1}/{\sqrt{n}}\right)$, which proves \eqref{Proof of theorem: main3}.
\end{proof}

\section{Supercritical case}\label{Section: proof super}

In this section, we turn to the supercritical case and prove Theorem \ref{Theorem: super LLN} and Theorem \ref{Theorem: super}.

\subsection{Approximating the coalescence times}

In the supercritical case, because the coalescent tree is star-shaped, it is more convenient to measure the coalescence times by looking forward from time zero rather than backwards from time $T$. We therefore let  $G_{i,n,T}=T-H_{i,n,T}$, which has density
$$f_{G_{i,n,T}|Y_{n,T}=y}(t)=\frac{y\lambda+(r-y\lambda)e^{-rT}}{y\lambda(1-e^{-rT})} \cdot \frac{y\lambda r^2 e^{-r(T-t)}}{(y\lambda+(r-y\lambda)e^{-r(T-t)})^2},\qquad t\in(0,T).$$
From \eqref{Recovering the branch lenght from coalescent point process: definition of L^k_0,n,T}, we get 
\begin{equation}\label{Proof of theorem super: formula for L^k_0,n,T}
L^{k}_{0,n,T}=\left(H_{k,n,T}-\max_{1\le i\le k-1}H_{i,n,T}\right)^{+}=\left(\min_{1\le i\le k-1}G_{i,n,T}-G_{k,n,T}\right)^{+}.
\end{equation}
Likewise, from \eqref{Recovering the branch lenght from coalescent point process: definition of L^k_i,n,T} and \eqref{Recovering the branch lenght from coalescent point process: definition of L^k_n-k,n,T}, we get
\begin{align}\label{Proof of theorem super: formula for L^k_i,n,T}
L^{k}_{i,n,T}&=\left(H_{i,n,T}\wedge H_{i+k,n,T}-\max_{i+1\le j\le i+k-1} H_{j,n,T}\right)^{+}\nonumber\\
&=\left(\min_{i+1\le j\le i+k-1} G_{j,n,T}- (G_{i,n,T}\vee G_{i+k,n,T})\right)^{+}\qquad \text{ for }1\le i\le n-k-1
\end{align}
and
\begin{equation}\label{Proof of theorem super: formula for L^k_n-k,n,T}
L^{k}_{n-k,n,T}=\left(H_{n-k,n,T}-\max_{n-k+1\le j\le n-1} H_{j,n,T}\right)^{+}=\left(\min_{n-k+1\le j\le n-1} G_{j,n,T}-G_{n-k,n,T}\right)^{+}.
\end{equation}
Arguments similar to those in Sections \ref{Section: approximation of H_i,n,T} and \ref{Section: error bounds} were applied to the supercritical case in \cite{johnson2023estimating}.  We state the relevant results in the supercritical case here.  The proofs can be found in Section 1.1 of the supplemental information in \cite{johnson2023estimating}. 

For large $n$ and $T_n$, the random variable $G_{i, n, T_n}$ is well approximated by $G_{i}$, which can be defined in the following way:
\begin{enumerate}
\item Choose $W$ from the exponential distribution with rate $1$.

\item Choose $U_{i}$ for $i=1,\dots, n-1$ independently from the logistic distribution, with density
$$
f_{U_{i}}(u)=\frac{e^u}{\left(1+e^u\right)^{2}}, \quad u \in(-\infty, \infty).
$$
\item Let $G_{i}=\frac{1}{r_n}\left(\log (1 / W)+U_{i}+\log n\right)$.
\end{enumerate}
\begin{remark}
The random variables $U_i$ represent the fluctuations in the coalescence times. This logistic distribution for the coalescence times also arises when each individual is sampled independently with probability $y$ and we let $y \rightarrow 0$. See equation (4.7) of \cite{ignatieva2020characterisation}.   
\end{remark}

The error from this approximation is bounded by the following lemma, whose proof can be found in Lemma 6 of the supplemental information in \cite{johnson2023estimating}.
\begin{lemma}\label{Lemma: corollory to the error bounds, preparation}
Assume condition \eqref{Introduction: sequence condition super LLN} holds. Then as $n\rightarrow\infty$,
 \begin{equation*}
\frac{r_n}{n}\sum_{i=1}^{n-1}\left|G_{i,n,T_n}-G_{i}\right|\stackrel{P}{\longrightarrow}0.
\end{equation*}
Assume condition \eqref{Introduction: sequence condition super} holds. Then as $n\rightarrow\infty$,
\begin{equation*}
 \frac{r_n}{\sqrt{n}}\sum_{i=1}^{n-1}\left|G_{i,n,T_n}-G_{i}\right|\stackrel{P}{\longrightarrow}0.
\end{equation*}
\end{lemma}

For $k\ge 2$, let 
\begin{equation} \label{Proof of theorem super: definition of L^k_i}
 L^k_i=\left(\min_{i+1\le j\le i+k-1} G_{j}- (G_{i}\vee G_{i+k}) \right)^{+},\qquad1\le i\le n-k-1 
\end{equation}
be the approximation of the length of the portion of the $i$th branch that supports $k$ leaves when $n$ and $T$ are large.  The next lemma shows that we can use the $L_i^k$ to approximate the quantities $L_{n,T_n}^k$.

\begin{lemma}\label{Lemma: corollory to the error bounds}
Suppose $k\ge 2$.
 Assume condition \eqref{Introduction: sequence condition super LLN} holds. Then as $n\rightarrow\infty$,
\begin{equation*}
\frac{r_n}{n}\left(L_{n,T_n}^k -\sum_{i=1}^{n-k-1}L^k_{i}\right)\stackrel{P}{\longrightarrow}0.
\end{equation*}
 Assume condition \eqref{Introduction: sequence condition super} holds. Then as $n\rightarrow\infty$,
\begin{equation*}
\frac{r_n}{\sqrt{n}}\left(L_{n,T_n}^k-\sum_{i=1}^{n-k-1}L^k_i\right)\stackrel{P}{\longrightarrow}0.
\end{equation*}
\end{lemma}

\begin{proof}
The proofs under conditions \eqref{Introduction: sequence condition super LLN} and \eqref{Introduction: sequence condition super} are similar. We present the proof under condition \eqref{Introduction: sequence condition super} here. Using \eqref{Recovering the branch length from the coalescent point process: definition of L^k_n,T}, we have
\begin{equation}\label{Proof of theorem super: corrollary to error bounds}
L_{n,T_n}^k - \sum_{i=1}^{n-k-1} L_i^k = \sum_{i=0}^{n-k}L^{k}_{i,n,T_n}-\sum_{i=1}^{n-k-1}L^k_i= L^{k}_{0,n,T_n} + L^k_{n-k,n,T_n}+\sum_{i=1}^{n-k-1} \left(L^{k}_{i,n,T_n}-L^k_{i}\right).
\end{equation}
For the third term in \eqref{Proof of theorem super: corrollary to error bounds}, using the triangle inequality along with \eqref{Proof of theorem super: formula for L^k_i,n,T} and \eqref{Proof of theorem super: definition of L^k_i}, we have
$$
\left|L^{k}_{i,n,T_n}-{L}^{k}_i\right|\le \sum_{j=i}^{i+k} |G_{j,n,T_n}-G_j|.
$$  
Summing over $1\le i\le n-k-1$ and using Lemma \ref{Lemma: corollory to the error bounds, preparation}, we have
\begin{equation}\label{Proof of theorem super: corrollary to error bounds, 1}
\frac{r_n}{\sqrt{n}}\sum_{i=1}^{n-k-1}\left(L^{k}_{i,n,T_n}-L^k_i\right)\stackrel{P}{\longrightarrow}0.
\end{equation}
For the first term in \eqref{Proof of theorem super: corrollary to error bounds}, we have 
\begin{equation}\label{Proof of theorem super: corrollary to error bounds, 2}
\begin{split}
L^{k}_{0,n,T_n}=\left(\min_{1\le i\le k-1}G_{i,n,T_n}-G_{k,n,T_n}\right)^{+}\le G_{1,n,T_n} \le G_{1}+|G_{1,n,T_n}-G_{1}|.
\end{split}
\end{equation}
By Lemma \ref{Lemma: corollory to the error bounds, preparation}, we have
\begin{equation}\label{Proof of theorem super: corrollary to error bounds, 3}
\frac{r_n}{\sqrt{n}}|G_{1,n,T_n}-G_{1}|\stackrel{P}{\longrightarrow}0.
\end{equation}
Also, 
\begin{equation}\label{Proof of theorem super: corrollary to error bounds, 4}
\frac{r_n}{\sqrt{n}} \cdot G_{1} =\frac{1}{\sqrt{n}}\left(\log (1 / W)+U_{i}+\log n\right) \stackrel{P}{\longrightarrow}0.
\end{equation}
Equations \eqref{Proof of theorem super: corrollary to error bounds, 2}, \eqref{Proof of theorem super: corrollary to error bounds, 3}, and \eqref{Proof of theorem super: corrollary to error bounds, 4} imply $L^k_{0,n,T_n} \stackrel{P}{\longrightarrow}0.$  The same argument gives $L^k_{n-k,n,T_n} \stackrel{P}{\longrightarrow}0.$  The lemma follows from these results along with \eqref{Proof of theorem super: corrollary to error bounds} and \eqref{Proof of theorem super: corrollary to error bounds, 1}.
\end{proof}

\subsection{Moment computations}

\begin{lemma}\label{Lemma: moment computation super}
For $k \geq 2$, let $L^k_i$ be defined as in \eqref{Proof of theorem super: definition of L^k_i}.  Then
$$
\mathbb{E}[L^k_i]=\frac{1}{r_n k(k-1)}.
$$
\end{lemma}
\begin{proof}
We have
$$
\begin{aligned}
\E[L^k_i]&=\mathbb{E}\left[\left(\min _{i+1 \leq j \leq i+k-1} G_{j}-\left(G_{i}\vee G_{i+k}\right)\right)^{+}\right] \\
&=\frac{1}{r_n} \mathbb{E}\left[\left(\min _{i+1 \leq j \leq i+k-1} U_{j }-(U_{i }\vee U_{i+k })\right)^{+}\right] \\
&=\frac{1}{r_n} \int_{-\infty}^{\infty}  \mathbb{P}\left(U_i\vee U_{i+k}\le x\le \min_{i+1\le j\le i+k-1} U_j\right)\ dx\\
&=\frac{1}{r_n}\int_{-\infty}^\infty \left(\frac{e^x}{1+e^x}\right)^2\left(\frac{1}{1+e^x}\right)^{k-1}\ dx\\
&=\frac{1}{r_n}\int_{-\infty}^{\infty} \frac{e^{2x}}{(1+e^x)^{k+1}}\ dx\\
&=\frac{1}{r_n k (k-1)},
\end{aligned}
$$
as claimed.
\end{proof}

We now consider the covariance computation, which is more complicated.  Let $2 \leq k' \leq k$.  We have
$$
\begin{aligned}
\E\left[L^k_iL^{k'}_{i'}\right]&= \mathbb{E}\left[\left(\min _{i+1 \leq j \leq i+k-1} G_{j}- (G_{i} \vee G_{i+k})\right)^{+}\left(\min _{i^{\prime}+1 \leq j^{\prime} \leq i^{\prime}+k^{\prime}-1} G_{j^{\prime}}-(G_{i^{\prime}} \vee G_{i^{\prime}+k^{\prime}})\right)^{+}\right] \\
&=\frac{1}{r_n^2} \mathbb{E}\left[\left(\min _{i+1 \leq j \leq i+k-1} U_{j}-(U_{i}\vee U_{i+k}) \right)^{+}\left(\min _{i^{\prime}+1 \leq j^{\prime} \leq i^{\prime}+k^{\prime}-1} U_{j'}- (U_{i'} \vee U_{i'+k'}) \right)^{+}\right] \\
&=\frac{1}{r_n^2} \int_{-\infty}^{\infty} \int_{-\infty}^{\infty} \mathbb{P}\left(U_{i }\vee U_{i+k }<y<\min _{i+1 \leq j \leq i+k-1} U_{j },\right.\\
&\qquad\qquad\qquad\qquad\qquad \left.U_{i' }\vee U_{i'+k' }<z<\min _{i^{\prime}+1 \leq j^{\prime} \leq i^{\prime}+k^{\prime}-1} U_{j' }\right) \ dy \ dz.
\end{aligned}
$$
Define
$$
F\left(i, i^{\prime}, k, k^{\prime}, y, z\right)=\mathbb{P}\left(U_{i } \vee U_{i+k }<y<\min _{i+1 \leq j \leq i+k-1} U_{j },\ U_{i'} \vee U_{i'+k' }<z<\min _{i^{\prime}+1 \leq j^{\prime} \leq i^{\prime}+k^{\prime}-1} U_{j' }\right).
$$
Then
\begin{align*}
&\Cov\left(\frac{r_n}{\sqrt{n}} \sum_{i=1}^{n-k-1}L^k_i, \frac{r_n}{\sqrt{n}}\sum_{i=1}^{n-k'-1} L^{k'}_i\right) \\
&\qquad=\frac{1}{n}\sum_{i=1}^{n-k-1}\sum_{i'=1}^{n-k'-1}\left(\int_{-\infty}^\infty\int_{-\infty}^\infty F(i,i',k,k',y,z)\ dy \ dz-\frac{1}{ k(k-1)k'(k'-1)}\right).
\end{align*}
Notice that $\Cov(L^k_i,L^{k'}_{i'})=0$ if $\{i,i+1,\dots,i+k\}$ is disjoint from $\{i',i'+1,\dots, i'+k'\}$.  Therefore, if we define
\begin{equation}\label{Vkk}
V_{k,k'} := \sum_{i'=i-k'}^{i+k} \bigg(\int_{-\infty}^\infty\int_{-\infty}^\infty F(i,i',k,k',y,z)\ dy \ dz \bigg) - \frac{k+k'+1}{ k(k-1)k'(k'-1)},
\end{equation}
where the choice of $i$ is arbitrary, then
\begin{equation}\label{CovLimit}
\lim_{n \rightarrow \infty} \Cov\left(\frac{r_n}{\sqrt{n}} \sum_{i=1}^{n-k-1}L^k_i, \frac{r_n}{\sqrt{n}}\sum_{i=1}^{n-k'-1} L^{k'}_i\right) = V_{k,k'}.
\end{equation}

To calculate the limiting covariances $V_{k,k'}$ when $2 \leq k' \leq k$, we need to evaluate $F\left(i, i^{\prime}, k, k^{\prime}, y, z\right)$.  Let $$F(x)=\frac{e^{x}}{1+e^{x}}$$ be the c.d.f. for the standard logistic distribution.  We consider the following eight cases:
\begin{enumerate}
\item Suppose $i' = i - k'$.  Observe that $F\left(i, i^{\prime}, k, k^{\prime}, y, z\right)$ is the probability that we have $U_{i'} < z$, $U_{i'+1}, \dots, U_{i'+k'-1} > z$, $U_i < z$, $U_i < y$, $U_{i+1}, \dots, U_{i+k-1} > y$, and $U_{i+k} < y$.  If $y < z$, then the condition $U_i < z$ can be removed, and if $z < y$, the condition $U_i < y$ can be removed.  Therefore,
\begin{equation}\label{F1}
F\left(i, i^{\prime}, k, k^{\prime}, y, z\right) = \left\{
\begin{array}{ll} F(y)^2 F(z) (1 - F(y))^{k-1} (1 - F(z))^{k'-1} & \mbox{ if } y < z, \\
F(y) F(z)^2 (1 - F(y))^{k-1} (1 - F(z))^{k'-1} & \mbox{ if } y > z. \end{array} \right.
\end{equation}

\item Suppose $i - k' < i' < i$.  Note that if $y < z$, then we can not have $U_i < y$ and $U_{i'+1}, \dots, U_{i'+k'-1} > z$ because $i' < i < i'+k'$.  If $y > z$, then we can not have both $U_{i'+k'} < z$ and $U_{i+1}, \dots, U_{i+k-1} > y$ because $i < i'+k' < i + k$.  Therefore, $$F\left(i, i^{\prime}, k, k^{\prime}, y, z\right) = 0.$$

\item Suppose $i' = i$ and $k' < k$.  Suppose first that $y < z$.  Then $F\left(i, i^{\prime}, k, k^{\prime}, y, z\right)$ is the probability that $U_i < y$, $U_{i+1}, \dots, U_{i+k'-1} > z$, $y < U_{i+k'} < z$, $U_{i+k'+1}, \dots, U_{i+k-1} > y$, and $U_{i+k} < y$.  If instead $y > z$, then the condition $y < U_{i+k'} < z$ can not hold.  Therefore,
\begin{equation}\label{F2}
F\left(i, i^{\prime}, k, k^{\prime}, y, z\right) = \left\{
\begin{array}{ll} F(y)^2(F(z) - F(y))(1 - F(y))^{k-k'-1}(1 - F(z))^{k'-1} & \mbox{ if } y < z, \\
0 & \mbox{ if } y > z. \end{array} \right.
\end{equation}

\item Suppose $i = i'$ and $k' = k$.  If $y < z$, then $F\left(i, i^{\prime}, k, k^{\prime}, y, z\right)$ is the probability that $U_i < y$, $U_{i+k} < y$, and $U_{i+1}, \dots, U_{i+k-1} > z$.  If $y > z$, then $F\left(i, i^{\prime}, k, k^{\prime}, y, z\right)$ is the probability that $U_i < z$, $U_{i+k} < z$, and $U_{i+1}, \dots, U_{i+k-1} > y$.  Therefore,
\begin{equation}\label{F3}
F\left(i, i^{\prime}, k, k^{\prime}, y, z\right) = \left\{
\begin{array}{ll} F(y)^2 (1 - F(z))^{k-1} & \mbox{ if } y < z, \\
F(z)^2(1 - F(y))^{k-1} & \mbox{ if } y > z. \end{array} \right.
\end{equation}

\item Suppose $i' > i$ and $i' + k' < i + k$, which is only possible when $k - k' \geq 2$.  Then $F\left(i, i^{\prime}, k, k^{\prime}, y, z\right)$ is the probability that $U_i < y$, $U_{i+k} < y$, $y < U_{i'} < z$, $y < U_{i'+k'} < z$, and the other values of $U_j$ for $i < j < i+k$ must be greater than $y$, and must also be greater than $z$ if $i' < j < i'+k'$.  This is only possible when $y < z$.  Therefore,
\begin{equation}\label{F4}
F\left(i, i^{\prime}, k, k^{\prime}, y, z\right) = \left\{
\begin{array}{ll} F(y)^2(F(z) - F(y))^2(1 - F(y))^{k-k'-2}(1 - F(z))^{k'-1} & \mbox{ if } y < z, \\
0 & \mbox{ if } y > z. \end{array} \right.
\end{equation}

\item Suppose $i' > i$ and $i' + k' = i + k$, which is only possible when $k > k'$.  Suppose first that $y < z$.  Then $F\left(i, i^{\prime}, k, k^{\prime}, y, z\right)$ is the probability that $U_i < y$, $U_{i+1}, \dots, U_{i'-1} > y$, $y < U_{i'} < z$, $U_{i'+1}, \dots, U_{i'+k'-1} > z$, and $U_{i+k} > y$.  If instead $y > z$, then the condition that $y < U_{i'} < z$ can not hold.  Therefore, noting that $i' - i = k - k'$, the expression for $F\left(i, i^{\prime}, k, k^{\prime}, y, z\right)$ is the same as in (\ref{F2}).

\item Suppose $i < i' < i + k < i' + k'$.  This is similar to Case 2.  Note that if $y < z$, then we can not have both $U_{i+k} < y$ and $U_{i'+1}, \dots, U_{i'+k'-1} > z$ because $i' < i+k < i'+k'$.  If $y > z$, then we can not have both $U_{i'} < z$ and $U_{i+1}, \dots, U_{i+k-1} > y$ because $i < i' < i+k$.  It follows that $F\left(i, i^{\prime}, k, k^{\prime}, y, z\right) = 0.$

\item Suppose $i' = i + k$.  This is similar to Case 1.  Here $F\left(i, i^{\prime}, k, k^{\prime}, y, z\right)$ is the probability that
$U_i < y$, $U_{i+1}, \dots, U_{i+k-1} > y$, $U_{i'} < y$, $U_{i'} < z$, $U_{i'+1}, \dots, U_{i'+k'-1} > z$, and $U_{i'+k'} < z$.
If $y < z$, then the condition $U_{i'} < z$ can be removed, and if $z < y$, the condition $U_i < y$ can be removed.  Therefore, the expression for $F\left(i, i^{\prime}, k, k^{\prime}, y, z\right)$ is the same as in (\ref{F1}).
\end{enumerate}

When $k = k'$, we have Case 4 above but not Cases 3, 5, or 6.  When $k = k'+1$, we do not have Cases 4 or 5, but we have Cases 3 and 6.  When $k \geq k'+2$, we have Cases 3, 5, and 6, but not Case 4.  Now let $I_{1,k,k'}$, $I_{2,k,k'}$, $I_{3,k}$, and $I_{4,k,k'}$ be the values of the double integral in (\ref{Vkk}) when the expression for $F\left(i, i^{\prime}, k, k^{\prime}, y, z\right)$ is given by (\ref{F1}), (\ref{F2}), (\ref{F3}), and (\ref{F4}) respectively.  Then
\begin{equation}\label{Vkkformula}
V_{k,k'} = - \frac{k+k'+1}{k(k-1)k'(k'-1)} + \left\{
\begin{array}{ll} 2I_{1,k,k'} + I_{3,k} & \mbox{ if } k = k' \\
2I_{1,k,k'} + 2I_{2,k,k'} & \mbox{ if } k = k' + 1 \\
2I_{1,k,k'} + 2I_{2,k,k'} + (k - k' - 1)I_{4,k,k'} & \mbox{ if } k \geq k' + 2,
\end{array} \right.
\end{equation}
where
\begin{align*}
I_{1,k,k'} &= \int_{-\infty}^{\infty} \int_{-\infty}^z \frac{e^{2y}}{(1 + e^y)^{k+1}} \cdot \frac{e^z}{(1 + e^z)^{k'}} \: dy \: dz + \int_{-\infty}^{\infty} \int_{-\infty}^y \frac{e^{y}}{(1 + e^y)^{k}} \cdot \frac{e^{2z}}{(1 + e^z)^{k'+1}} \: dz \: dy \\
I_{2,k,k'} &= \int_{-\infty}^{\infty} \int_{-\infty}^z \frac{e^{2y}}{(1 + e^y)^{k-k'+1}} \cdot \frac{1}{(1 + e^z)^{k'-1}} \cdot \bigg( \frac{e^z}{1 + e^z} - \frac{e^y}{1+ e^y} \bigg) \: dy \: dz \\
I_{3,k} &= \int_{-\infty}^{\infty} \int_{-\infty}^z \frac{e^{2y}}{(1 + e^y)^2} \cdot \frac{1}{(1 + e^z)^{k-1}} \: dy \: dz + \int_{-\infty}^{\infty} \int_{-\infty}^y \frac{1}{(1 + e^y)^{k-1}} \cdot \frac{e^{2z}}{(1 + e^z)^2} \: dz \: dy \\
I_{4,k,k'} &= \int_{-\infty}^{\infty} \int_{-\infty}^z \frac{e^{2y}}{(1 + e^y)^{k-k'}} \cdot \frac{1}{(1 + e^z)^{k'-1}} \cdot \bigg( \frac{e^z}{1 + e^z} - \frac{e^y}{1+ e^y} \bigg)^2 \: dy \: dz.
\end{align*}
The evaluation of these integrals is tedious, and we defer details to the appendix.  One can compute
\begin{equation}\label{I1final}
I_{1,k,k'} = \frac{1}{(k-1)(k'-1)(k+k'-1)},
\end{equation}
\begin{equation}\label{I2final}
I_{2, k, k'} = \left\{
\begin{array}{ll}\frac{1}{2(k-2)} + \frac{1}{2(k-1)} - \frac{\pi^2}{6} + \sum_{j=1}^{k-2} \frac{1}{j^2} & \mbox{ if } k = k' + 1 \\
\frac{k+k'-2}{(k'-1)(k-1)(k-k')(k-k'+1)} - \frac{2}{(k-k'-1)(k-k')(k-k'+1)} \sum_{j=k'}^{k-2} \frac{1}{j} & \mbox{ if } k \geq k' + 2, \end{array} \right.
\end{equation}
\begin{equation}\label{I3final}
I_{3,k} = \frac{\pi^2}{3} - \frac{2}{k-1} - 2 \sum_{j=1}^{k-2} \frac{1}{j^2},
\end{equation}
and
\begin{equation}\label{I4final}
I_{4, k, k'} = \left\{
\begin{array}{ll} \frac{1}{6(k-3)} - \frac{5}{6(k-2)} - \frac{1}{3(k-1)} + \frac{\pi^2}{6} - \sum_{j=1}^{k-2} \frac{1}{j^2} & \mbox{ if } k = k' + 2, \\
\frac{1}{(k'-1)(k-k')(k-k'+1)} + \frac{6}{(k-k'-2)(k-k'-1)(k-k')(k-k'+1)} \sum_{j=k'}^{k-2} \frac{1}{j} \\
\qquad \qquad - \frac{1}{k'(k-k'-2)(k-k'-1)} - \frac{2}{(k-1)(k-k'-1)(k-k')(k-k'+1)} & \mbox{ if } k \geq k' +3. \end{array} \right.
\end{equation}
The formulas for $V_{k,k'}$ in \eqref{finalcov} can be obtained from \eqref{Vkkformula}, \eqref{I1final}, \eqref{I2final}, \eqref{I3final}, and \eqref{I4final}, along with some tedious algebra.

\subsection{Proof of Theorem  \ref{Theorem: super LLN}}
\begin{proof}
Let $k \geq 2$.  By Lemma \ref{Lemma: corollory to the error bounds}, it suffices to show that as $n \rightarrow \infty$,
\begin{equation}\label{superLLNsuff}
\frac{r_n}{n}\sum_{i=1}^{n-k-1} L^{k}_i\stackrel{P}{\longrightarrow}\frac{1}{k(k-1)}.
\end{equation}
By Lemma \ref{Lemma: moment computation super} and \eqref{CovLimit}, we have
\begin{equation*}
\begin{split}
&\lim_{n\rightarrow\infty}\E\left[\left(\frac{r_n}{n}\sum_{i=1}^{n-k-1} L^{k}_i-\frac{1}{k(k-1)}\right)^2\right]\\
&\qquad=\lim_{n\rightarrow\infty}\Var\left(\frac{r_n}{n}\sum_{i=1}^{n-k-1} L^{k}_i\right)+\lim_{n\rightarrow\infty}\left(\frac{n-k-1}{ nk(k-1)}-\frac{1}{k(k-1)}\right)^2\\
&\qquad=\lim_{n\rightarrow\infty}\frac{1}{n}V_{k,k} \\
&\qquad=0.
\end{split}
\end{equation*}
The result \eqref{superLLNsuff} now follows from Chebyshev's inequality, which completes the proof.
\end{proof}

\subsection{Proof of Theorem \ref{Theorem: super}}

\begin{proof}
Let $V$ be a covariance matrix whose entries are given by \eqref{Vkkformula}.  We will show that 
\begin{equation}\label{Proof of theorem super: sequnce}
\left\{\frac{r_n}{\sqrt{n}}\sum_{i=1}^{n-k-1} \left(L^k_i-\E[ L^k_i]\right)\right\}_{k=2}^{K} \Rightarrow N(0,V).
\end{equation}
We proceed as in the proof of Theorem \ref{Theorem}. Let $\left\{x_{k}\right\}_{k=2}^{K}$ be an arbitrary sequence of nonzero real numbers and consider the linear combination
\begin{equation}\label{Proof of theorem super: linear combination}
\sum_{k=2}^{K} x_{k} r_n \sum_{i=1}^{n-K-1}\left( L^k_i-\E[ L^k_i]\right)=\sum_{i=1}^{n-K-1}\sum_{k=2}^{K}x_k r_n \left( L^k_i-\E[ L^k_i]\right).
\end{equation}
 Let 
\begin{align}\label{Proof of theorem super: definition of L_i}
L_i&=\sum_{k=2}^{K} x_k r_n \left(L^k_i-\E[L^k_i]\right) \nonumber \\
&=\sum_{k=2}^K x_k  \left(\left(\min_{i+1\le j\le i+k-1} U_{j}-U_{i} \vee U_{i+k} \right)^+ -\E\left[\left(\min_{i+1\le j\le i+k-1} U_{j}-U_{i} \vee U_{i+k}\right)^+\right]\right).
\end{align}
Note that $L_i=f(U_{i },\dots,U_{i+K })$ for some continuous function $f$, which depends on the constants $x_k$ but not on $r_n$.
We apply the $m$-dependent central limit theorem for a sequence of random variables to the random variables $L_i$ (see Theorem 3 in \cite{diananda1955central}). We need to check
\begin{equation}\label{Proof of theorem super: clt first condition}\liminf_{n\rightarrow\infty} \frac{1}{n}\Var\left(\sum_{i=1}^{n-K-1}L_i\right)>0,
\end{equation}
and 
\begin{equation}\label{Proof of theorem super: clt second condition}
\lim_{n\rightarrow\infty}\frac{1}{n} \sum_{i=1}^{n-K-1} \E\left[L_i^2\mathbbm{1}_{\left\{|L_i|>\epsilon\sqrt{n}\right\}}\right]=0.
\end{equation}

Assume for now that \eqref{Proof of theorem super: clt first condition} and \eqref{Proof of theorem super: clt second condition} hold.  Then the distribution of the expression in \eqref{Proof of theorem super: linear combination} is asymptotically normal.  Since $\left\{x_{k}\right\}_{k=2}^{K}$ is arbitrary, we obtain \eqref{Proof of theorem super: sequnce} from the Cram\'{e}r-Wold theorem and \eqref{CovLimit}.  To prove Theorem \ref{Theorem: super}, we need to show that \eqref{Proof of theorem super: sequnce} implies that
\begin{equation}\label{NTS1.4}
\left\{\frac{r_n}{\sqrt{n}}\left(L_{n,T_n}^k-\frac{n}{r_nk(k-1)}\right)\right\}_{k=2}^K \Rightarrow N(0,V).
\end{equation}
By Lemma \ref{Lemma: moment computation super}, we have
\begin{equation*}
\E[L^k_i]=\frac{1}{r_n k(k-1)}.
\end{equation*}
It follows that for $k \geq 2$,
\begin{equation}\label{meanapprox}
\lim_{n\rightarrow\infty}\frac{r_n}{\sqrt{n}}\left(\sum_{i=1}^{n-k-1} \E[L^k_i]-\frac{n}{r_n k(k-1)}\right)=\lim_{n\rightarrow\infty}\frac{k+1}{\sqrt{n}k(k-1)}=0.
\end{equation}
The result \eqref{NTS1.4} follows from \eqref{Proof of theorem super: sequnce} and \eqref{meanapprox} along with Lemma \ref{Lemma: corollory to the error bounds}.

It remains to show that \eqref{Proof of theorem super: clt first condition} and \eqref{Proof of theorem super: clt second condition} hold.  For condition \eqref{Proof of theorem super: clt first condition}, recall from \eqref{Proof of theorem super: definition of L_i} that $L_i=f(U_{i },\dots,U_{i+K })$ for some continuous function $f$. Let $\mathcal{A}_{3Kl}$ be the event that
$U_{3Kl }\ge U_{3Kl-1 }\vee U_{3Kl+1 }$ for $l=1,2\dots$. Let $\mathcal{G}$ be the $\sigma$ field generated by $\{U_{i}: i \text{ is not a  multiple of }3K\}_{i=1}^{\infty}$ and the events $\{\mathcal{A}_{3Kl}\}_{l=1}^{\infty}$. Consider the index set 
$$
I_l=\{3Kl-j: j=0,1,2,\dots K\},\qquad l=1,2,\dots
$$
then $\{L_i\}_{i\notin \cup I_l}$ is measurable with respect to $\mathcal{G}$. Also, for $l\neq l'$, $\{L_i\}_{i\in I_l}$ is conditionally independent of $\{L_i\}_{i\in I_{l'}}$ given $\mathcal{G}$. Therefore, 
\begin{equation*}
\begin{split}
\liminf_{n\rightarrow\infty}\frac{1}{n}\Var\left(\sum_{i=1}^{n-K-1}L_i\right)&\ge\liminf_{n\rightarrow\infty}\frac{1}{n} \E\left[\Var\left(\sum_{i=1}^{n-K-1}L_i\Big|\mathcal{G}\right)\right]\\
&\ge\liminf_{n\rightarrow\infty}\frac{1}{n}\E\left[\Var\left(\sum_{l=1}^{\left\lfloor \frac{n-K-1}{3K}\right\rfloor}\sum_{i\in I_l}L_i\Big|\mathcal{G}\right)\right]\\
&=\liminf_{n\rightarrow\infty}\frac{1}{n}\left\lfloor \frac{n-K-1}{3K}\right\rfloor\E\left[\Var\left(\sum_{i\in I_1} L_i\Big|\mathcal{G}\right)\right]\\
&=\frac{1}{3K}\E\left[\Var\left(\sum_{i\in I_1} L_i\Big|\mathcal{G}\right)\right].
\end{split}
\end{equation*}
Therefore, it suffices to show that there exists some constant $\epsilon > 0$ depending only on $x_2,\dots,x_k$ such that
\begin{equation}\label{Proof: sufficient condition for the variance to be order n}
\E\left[\Var\left(\sum_{i\in I_1} L_i\Big|\mathcal{G}\right)\right]\ge\epsilon.
\end{equation}
Now we restrict ourselves to the event $\mathcal{A}_{3K}$. Recall that $\mathcal{A}_{3K}\in\mathcal{G}$, we have
\begin{equation}\label{Proof: conditional variance lower bound}
\begin{split}
\E\left[\Var\left(\sum_{i\in I_1} L_i\Big|\mathcal{G}\right)\right]\ge
\E\left[\Var\left(\sum_{i\in I_1} L_i\Big|\mathcal{G}\right)\mathbbm{1}_{\mathcal{A}_{3K}}\right]=\E\left[\Var\left(\sum_{i\in I_1} L_i\mathbbm{1}_{\mathcal{A}_{3K}}\Big|\mathcal{G}\right)\right].
\end{split}
\end{equation}
The key observation is that $x_k r_n L^k_i\mathbbm{1}_{\mathcal{A}_{3K}}$ is $\mathcal{G}$-measurable for $k\ge 3$. To see why this is true, recall from the definition of $L^k_i$ in \eqref{Proof of theorem super: definition of L^k_i} and the definition of $G_{i}$ that
\begin{equation*}
\begin{split}
x_k r_n L^k_i\mathbbm{1}_{\mathcal{A}_{3K}}=x_k\left(\min_{i+1\le j\le i+k-1}U_{j}-U_{i}\vee U_{i+k }\right)^{+}\mathbbm{1}_{\mathcal{A}_{3K}}.
\end{split}
\end{equation*}
On the event $\mathcal{A}_{3K}$, the following hold:
\begin{enumerate}
\item If $i=3K$ or $i+k=3K$, then 
$$
\min_{i+1\le j\le i+k-1}U_{j}\le U_{i }\vee U_{i+k }
$$ 
and $L^k_i\mathbbm{1}_{\mathcal{A}_{3K}}$ is therefore 0. 
\item If $i+1\le 3K\le i+k-1$, then since $k\ge3$, either $3K-1$ or $3K+1$ is between $i+1$ and $i+k-1$. Since the event $\mathcal{A}_{3K}$ occurs, we have
$$
\left(\min_{i+1\le j\le i+k-1}U_{i }\right)\mathbbm{1}_{\mathcal{A}_{3K}}=\left(\min_{i+1\le j\le i+k-1,i\neq 3K}U_{i }\right)\mathbbm{1}_{\mathcal{A}_{3K}}.
$$
Therefore, $L^k_i\mathbbm{1}_{\mathcal{A}_{3K}}$ is $\mathcal{G}$-measurable. 
\item If $3K\notin\{i,i+1,\dots, i+k\}$, then $L^k_i$ is $\mathcal{G}$-measurable.
\end{enumerate}
For $k=2$, the arguments in case 1 and case 3 also apply. Therefore, the only case when $L^k_i\mathbbm{1}_{\mathcal{A}_{3K}}$ could fail to be $\mathcal{G}$-measurable is when $k=2$ and $i=3K-1$. Therefore, the last term in \eqref{Proof: conditional variance lower bound} is 
$$
\E\left[\Var\left(\sum_{i\in I_1} L_i\mathbbm{1}_{\mathcal{A}_{3K}}\Big|\mathcal{G}\right)\right]
=x_2^2\E\left[\Var\left(\left(U_{3K }-U_{3K-1 }\vee U_{3K+1 }\right)^{+}\mathbbm{1}_{\mathcal{A}_{3K}}\Big|\mathcal{G}\right)\right].
$$
Note that conditional on $\mathcal{G}$, $\left(U_{3K }-U_{3K-1 }\vee U_{3K+1 }\right)^{+}\mathbbm{1}_{\mathcal{A}_{3K}}$ is not almost surely a constant because the probability that $U_{3K }>M$ is strictly positive for all $M$, so the claim \eqref{Proof: sufficient condition for the variance to be order n} is proved.

For condition \eqref{Proof of theorem super: clt second condition}, note that \eqref{Proof of theorem super: definition of L_i} implies that $\E[L_i^2]<\infty$ provided that $\E[U_i^2]<\infty$, which is true because $U_i$ has a logistic distribution. Applying the dominated convergence theorem, we have
\begin{equation*}
\begin{split}
\lim_{n\rightarrow\infty}\frac{1}{n} \sum_{i=1}^{n-K-1} \E\left[L_i^2\mathbbm{1}_{\left\{|L_i|>\epsilon\sqrt{n}\right\}}\right]=\lim_{n\rightarrow\infty} \E\left[L_i^2\mathbbm{1}_{\left\{|L_i|>\epsilon\sqrt{n}\right\}}\right]=0,
\end{split}
\end{equation*}
which proves \eqref{Proof of theorem super: clt second condition}.
\end{proof}

\bigskip
\noindent {\bf {\Large Appendix -- Evaluation of the integrals}}

\bigskip
We explain here how to evaluate the four integrals $I_{1,k,k'}$, $I_{2,k,k'}$, $I_{3,k}$, and $I_{4,k,k'}$, that are needed to calculate the covariance in (\ref{Vkkformula}).  We begin by computing $I_{1,k,k'}$.  Making the substitution $u = 1 + e^y$, we get
\begin{equation}\label{usub}
\int_{-\infty}^z \frac{e^{2y}}{(1 + e^y)^{k+1}} \: dy = \int_1^{1 + e^z} \frac{u - 1}{u^{k+1}} \: du = \int_1^{1 + e^z} \frac{1}{u^k} \: du - \int_1^{1+e^z} \frac{1}{u^{k+1}} \: du.
\end{equation}
Therefore, the first integral in the expression for $I_{1,k,k'}$ becomes
\begin{align}\label{I1int}
&\int_{-\infty}^{\infty} \int_{-\infty}^z \frac{e^{2y}}{(1 + e^y)^{k+1}} \cdot \frac{e^z}{(1 + e^z)^{k'}} \: dy \: dz \nonumber \\
&\qquad= \int_{-\infty}^{\infty} \frac{e^z}{(1 + e^z)^{k'}} \bigg(\int_1^{1+e^z} \frac{1}{u^k} \: du \bigg) \: dz - \int_{-\infty}^{\infty} \frac{e^z}{(1 + e^z)^{k'}} \bigg( \int_1^{1+e^z} \frac{1}{u^{k+1}} \: du \bigg) \: dz.
\end{align}
By making the substitution $u = 1 + e^x$, one can easily check that if $j > 1$, then
\begin{equation}\label{int1}
\int_{-\infty}^{\infty} \frac{e^z}{(1 + e^z)^j} \: dz = \frac{1}{j-1}.
\end{equation}
It follows that if $j > 1$ and $m > 1$, then
\begin{equation}\label{int4}
\int_{-\infty}^{\infty} \frac{e^z}{(1 + e^z)^j} \bigg( \int_1^{1 + e^z} \frac{1}{u^m} \: du \bigg) \: dz = \frac{1}{m-1} \bigg(\frac{1}{j-1} - \frac{1}{j+m-2} \bigg) = \frac{1}{(j-1)(j+m-2)}.
\end{equation}
Therefore, using (\ref{I1int}) and (\ref{int4}), the first integral in the expression for $I_{1,k,k'}$ becomes
\begin{align}
\int_{-\infty}^{\infty} \int_{-\infty}^z \frac{e^{2y}}{(1 + e^y)^{k+1}} \cdot \frac{e^z}{(1 + e^z)^{k'}} \: dy \: dz 
&= \frac{1}{(k'-1)(k'+k-2)} - \frac{1}{(k'-1)(k'+k-1)} \nonumber \\
&= \frac{1}{(k'-1)(k+k'-1)(k+k'-2)}.
\end{align}
The second integral in the expression for $I_{1,k,k'}$ is the same but with the roles of $k$ and $k'$ reversed, so we get
$$\int_{-\infty}^{\infty} \int_{-\infty}^y \frac{e^{y}}{(1 + e^y)^{k}} \cdot \frac{e^{2z}}{(1 + e^z)^{k'+1}} \: dz \: dy = \frac{1}{(k-1)(k+k'-1)(k+k'-2)}.$$
Summing these two integrals gives the result in (\ref{I1final}).

We next compute $I_{3,k}$.  The two integrals in the expression for $I_{3,k}$ are the same, with the roles of $y$ and $z$ reversed, so we only need to compute one integral.  First, by making the substitution $u = 1 + e^y$, we get
$$\int_{-\infty}^{z} \frac{e^{2y}}{(1 + e^y)^2} \: dy = \int_1^{1 + e^z} \frac{w-1}{w^2} \: dw = \log(1 + e^z) - \frac{e^z}{1 + e^z}.$$
It follows that
\begin{equation}\label{I3prelim}
I_{3,k} = 2 \int_{-\infty}^{\infty} \frac{1}{(1 + e^z)^{k-1}} \bigg(\log(1 + e^z) - \frac{e^z}{1 + e^z} \bigg) \: dz.
\end{equation}
Note that if $j$ is a positive integer, then making the substitution $u = 1/(1 + e^x)$ and writing $1/(1-u)$ as an infinite series gives
\begin{equation}\label{logint}
\int_{-\infty}^{\infty} \frac{\log(1 + e^z)}{(1 + e^z)^j} \: dz = -\sum_{m=0}^{\infty} \int_0^1 u^{j+m-1} \log u \: du = \sum_{m=j}^{\infty} \frac{1}{m^2} = \frac{\pi^2}{6} - \sum_{m=1}^{j-1} \frac{1}{m^2}.
\end{equation}
By applying (\ref{int1}) and (\ref{logint}), we get (\ref{I3final}) from (\ref{I3prelim}).

We now consider $I_{2,k,k'}$.  Because this term does not appear when $k = k'$, we assume that $k > k'$.  Using that
\begin{equation}\label{rewrite}
\frac{e^z}{1 + e^z} - \frac{e^y}{1+ e^y} = \frac{1}{1 + e^y} - \frac{1}{1 + e^z},
\end{equation}
we can split the integral into two terms and write
$$I_{2,k,k'} = \int_{-\infty}^{\infty} \int_{-\infty}^z \frac{e^{2y}}{(1 + e^y)^{k-k'+2}} \cdot \frac{1}{(1 + e^z)^{k'-1}} \: dy \: dz - \int_{-\infty}^{\infty} \int_{-\infty}^z \frac{e^{2y}}{(1 + e^y)^{k-k'+1}} \cdot \frac{1}{(1 + e^z)^{k'}} \: dy \: dz.$$
Then using \eqref{usub},
\begin{align}\label{I2prelim}
I_{2,k,k'} &= \int_{-\infty}^{\infty} \frac{1}{(1 + e^z)^{k'-1}} \bigg( \int_1^{1+e^z} \frac{1}{u^{k-k'+1}} \: du - \int_1^{1+e^z} \frac{1}{u^{k-k'+2}} \: du \bigg) \: dz \nonumber \\
&\qquad \qquad - \int_{-\infty}^{\infty} \frac{1}{(1 + e^z)^{k'}} \bigg( \int_1^{1+e^z} \frac{1}{u^{k-k'}} \: du - \int_1^{1+e^z} \frac{1}{u^{k-k'+1}} \: du \bigg) \: dz.
\end{align}
Suppose first that $k \geq k'+2$. For positive integers $m$ and $n$ with $m \leq n$, let $$h(m,n) = \frac{1}{m} + \frac{1}{m+1} + \dots + \frac{1}{n}.$$ Observe that for positive integers $j$ and $m$ with $m \geq 2$, we have, making the substitution $v = 1 + e^z$,
\begin{align}\label{hint}
\int_{-\infty}^{\infty} \frac{1}{(1 + e^z)^j} \bigg( \int_1^{1+e^z} \frac{1}{u^m} \: du \bigg) \: dz &= \frac{1}{m-1} \int_{-\infty}^{\infty} \frac{(1+e^z)^{m-1} - 1}{(1+e^z)^{j+m-1}} \: dz \nonumber \\
&= \frac{1}{m-1} \int_1^{\infty} \frac{v^{m-1} - 1}{v^{j+m-1}(v-1)} \: dv \nonumber \\
&= \frac{1}{m-1} \int_1^{\infty} \frac{1}{v^{j+m-1}} \sum_{\ell=0}^{m-2} v^\ell \: dv \nonumber \\
&= \frac{h(j, j+m-2)}{m-1}.
\end{align}
Therefore,
\begin{align}\label{I2terms}
I_{2,k,k'} &= \frac{h(k'-1, k-2)}{k-k'} - \frac{h(k'-1, k-1)}{k-k'+1} - \frac{h(k', k-2)}{k-k'-1} + \frac{h(k',k-1)}{k-k'} \nonumber \\
&= h(k', k-2) \bigg( \frac{2}{k-k'} - \frac{1}{k-k'+1} - \frac{1}{k-k'-1} \bigg) \nonumber \\
&\qquad + \frac{1}{(k'-1)(k-k')} - \frac{1}{(k'-1)(k-k'+1)} - \frac{1}{(k-1)(k-k'+1)} + \frac{1}{(k-1)(k-k')} \nonumber \\
&= \frac{k+k'-2}{(k'-1)(k-1)(k-k')(k-k'+1)} - \frac{2h(k', k-2)}{(k-k'-1)(k-k')(k-k'+1)}.
\end{align}
Now suppose instead that $k = k' + 1$.  Then the third of the four terms in (\ref{I2prelim}) needs to be evaluated using (\ref{logint}) instead of (\ref{hint}), and we get
\begin{align}\label{I2special}
I_{2,k,k'} &= \frac{1}{k-2} - \frac{1}{2} \bigg(\frac{1}{k-2} + \frac{1}{k-1} \bigg) - \bigg( \frac{\pi^2}{6} - \sum_{j=1}^{k-2} \frac{1}{j^2} \bigg) + \frac{1}{k-1} \nonumber \\
&= \frac{1}{2(k-2)} + \frac{1}{2(k-1)} - \frac{\pi^2}{6} + \sum_{j=1}^{k-2} \frac{1}{j^2}.
\end{align}
The result in (\ref{I2final}) comes from (\ref{I2terms}) and (\ref{I2special}).

It remains to compute $I_{4,k,k'}$.  For this integral, we may assume that $k \geq k' + 2$.  We use (\ref{rewrite}), expand out the square, and break the integral into three terms to get
\begin{align*}
I_{4,k,k'} &= \int_{-\infty}^{\infty} \int_{-\infty}^z \frac{e^{2y}}{(1 + e^y)^{k-k'+2}} \cdot \frac{1}{(1 + e^z)^{k'-1}} \: dy \: dz \\
&\qquad- 2 \int_{-\infty}^{\infty} \int_{-\infty}^z \frac{e^{2y}}{(1 + e^y)^{k-k'+1}} \cdot \frac{1}{(1 + e^z)^{k'}} \: dy \: dz \\
&\qquad \qquad + \int_{-\infty}^{\infty} \int_{-\infty}^z \frac{e^{2y}}{(1 + e^y)^{k-k'}} \cdot \frac{1}{(1 + e^z)^{k'+1}} \: dy \: dz. 
\end{align*}
Then using (\ref{usub}), we get
\begin{align*}
I_{4,k,k'} &= \int_{-\infty}^{\infty} \frac{1}{(1 + e^z)^{k'-1}} \bigg( \int_1^{1 + e^z} \frac{1}{u^{k-k'+1}} \: du - \int_1^{1+e^z} \frac{1}{u^{k-k'+2}} \: du \bigg) \: dz \\
&\qquad -2 \int_{-\infty}^{\infty} \frac{1}{(1 + e^z)^{k'}} \bigg( \int_1^{1 + e^z} \frac{1}{u^{k-k'}} \: du - \int_1^{1+e^z} \frac{1}{u^{k-k'+1}} \: du \bigg) \: dz \\
&\qquad \qquad + \int_{-\infty}^{\infty} \frac{1}{(1 + e^z)^{k'+1}} \bigg( \int_1^{1 + e^z} \frac{1}{u^{k-k'-1}} \: du - \int_1^{1+e^z} \frac{1}{u^{k-k'}} \: du \bigg) \: dz.
\end{align*}
Suppose first that $k \geq k'+3$.  Then from (\ref{hint}), we get
\begin{align}\label{I4main}
I_{4,k,k'} &= \frac{h(k'-1, k-2)}{k-k'} - \frac{h(k'-1,k-1)}{k-k'+1} - \frac{2h(k',k-2)}{k-k'-1} \nonumber \\
&\qquad \qquad + \frac{2h(k',k-1)}{k-k'} + \frac{h(k'+1,k-2)}{k-k'-2} - \frac{h(k'+1,k-1)}{k-k'-1} \nonumber \\
&= h(k',k-2) \bigg(\frac{1}{k-k'-2} - \frac{3}{k-k'-1} + \frac{3}{k-k'} - \frac{1}{k-k'+1} \bigg) \nonumber \\
&\qquad \qquad+ \frac{1}{k-k'} \cdot \frac{1}{k'-1} - \frac{1}{k-k'+1} \bigg( \frac{1}{k'-1} + \frac{1}{k-1} \bigg) + \frac{2}{k-k'} \cdot \frac{1}{k-1} \nonumber \\
&\qquad \qquad+ \frac{1}{k-k'-2} \bigg( - \frac{1}{k'} \bigg) - \frac{1}{k-k'-1} \bigg( -\frac{1}{k'} + \frac{1}{k-1} \bigg) \nonumber \\
&= \frac{6 h(k', k-2)}{(k-k'-2)(k-k'-1)(k-k')(k-k'+1)} + \frac{1}{(k'-1)(k-k')(k-k'+1)} \nonumber \\
&\qquad \qquad - \frac{1}{k'(k-k'-2)(k-k'-1)} - \frac{2}{(k-1)(k-k'-1)(k-k')(k-k'+1)}.
\end{align}
When $k = k'+2$, we instead use (\ref{logint}) for the fifth term and simplify the other terms to get
\begin{align}\label{I4special}
I_{4,k,k'} &= \frac{1}{2} \bigg( \frac{1}{k-3} + \frac{1}{k-2} \bigg) - \frac{1}{3} \bigg( \frac{1}{k-3} + \frac{1}{k-2} + \frac{1}{k-1} \bigg) \nonumber \\
&\qquad \qquad- \frac{2}{k-2} + \bigg( \frac{1}{k-2} + \frac{1}{k-1} \bigg) + \frac{\pi^2}{6} - \sum_{j=1}^{k-2} \frac{1}{j^2} - \frac{1}{k-1} \nonumber \\
&= \frac{1}{6(k-3)} - \frac{5}{6(k-2)} - \frac{1}{3(k-1)} + \frac{\pi^2}{6} - \sum_{j=1}^{k-2} \frac{1}{j^2}.
\end{align}
The result (\ref{I4final}) comes from (\ref{I4main}) and (\ref{I4special}).

\bigskip
\noindent {\bf {\Large Acknowledgments}}

\bigskip
\noindent The authors thank Kit Curtius and Brian Johnson for helpful discussions related to this work, particularly in the supercritical case.

\bibliographystyle{plain}
\bibliography{Ref}

\end{document}